\documentclass{article}%
\usepackage{amsmath, amsthm}
\usepackage{amsfonts}
\usepackage{amssymb}
\usepackage{enumerate}
\usepackage{graphicx}
\usepackage[table]{xcolor}
\usepackage[all,arc,curve,color,frame]{xy} % til at tegne grafer og diagrammer

\newtheorem{theorem}{Theorem}[section]

\newtheorem{corollary}[theorem]{Corollary}

\newtheorem{definition}[theorem]{Definition}
\newtheorem{example}[theorem]{Example}

\newtheorem{lemma}[theorem]{Lemma}

\newtheorem{proposition}[theorem]{Proposition}
\newtheorem{remark}[theorem]{Remark}

\newcommand{\SP}{$C^*$-stable}

\newcommand{\WSP}{weakly $C^*$-stable}

\newcommand{\MWSP}{matricially stable}

\newcommand{\CC}{\mathbb C}
\newcommand{\FF}{\mathbb F}
\newcommand{\TT}{\mathbb T}
\newcommand{\II}{\mathbb I}
\newcommand{\NN}{\mathbb N}
\newcommand{\QQ}{\mathbb Q}
\newcommand{\RR}{\mathbb R}
\newcommand{\ZZ}{\mathbb Z}
\renewcommand{\phi}{\varphi}
\newcommand{\crystal}[1]{\mbox{${\mathfrak{#1}}$}}
\newcommand{\baso}{\mathsf{BS}}
\newcommand{\mcomm}[2]{\left\{#1,#2\right\}}
\newcommand{\acomm}[2]{\left[#1, #2\right]}
\newcommand{\matrM}[1]{\mathbf{M}_{#1}}
\newcommand{\myu}{\mathsf 1}
\newcommand{\cA}{\mathcal A}
\newcommand{\setof}[2]{\left\{ #1 \; \middle\vert \; #2 \right\}}
\newcommand{\genrel}[2]{\left\langle #1 \; \middle\vert \; #2 \right\rangle}

\DeclareMathOperator{\diag}{diag}
\DeclareMathOperator{\ind}{ind}
\DeclareMathOperator{\Ind}{Ind}

\DeclareMathOperator{\rank}{rank}
\DeclareMathOperator{\Span}{span}
\DeclareMathOperator{\tr}{tr}
\DeclareMathOperator{\wind}{wind}

\newcommand{\SESTART}{}%\color{red}}
\newcommand{\SEEND}{}%\color{black}}
\renewcommand\vec{\mathbf}

\begin{document}

\title{$C^*$-stability of discrete groups}
\author{S\o ren Eilers \and Tatiana Shulman \and Adam P.W. S\o rensen}
\date{\today}
\maketitle

\begin{abstract}
A group may be considered \emph{$C^*$-stable} if almost representations of the group in a $C^*$-algebra are always close to actual representations.
We initiate a systematic study of which discrete groups are $C^*$-stable or only stable with respect to some subclass of $C^*$-algebras, e.g. finite dimensional $C^*$-algebras.
We provide criteria and invariants for stability of groups and
this allows us to completely determine stability/non-stability of crystallographic groups, surface groups, virtually free groups, and certain Baumslag-Solitar groups.
We also show that among the non-trivial finitely generated torsion-free 2-step nilpotent groups the only $C^*$-stable group is $\ZZ$.
\end{abstract}

\tableofcontents

\section{Introduction}

The notion of stability appears in many forms throughout mathematics, where it often expresses how much being a little bit wrong matters.
Following Hyers and Ulam (\cite{dhh:slfe}) a general sense of this notion can be expressed as follows:
Are elements that ``almost'' satisfy some equation ``close'' to elements that exactly satisfy the equation?
In applied mathematics such a question is important since we should not expect computer models, or indeed the real world, to give answers that precisely match what our theory predicts.

As an example of a concrete stability problem consider the question of whether a set of almost commuting matrices must be close to a set of exactly commuting matrices.
The answer is strongly dependent on the classes of matrices one considers and  which matrix norm one uses to measure ``almost'' and ``close".
When using the operator norm the question is due to Halmos (\cite{Halmos}) and it has a positive answer for pairs of self-adjoint contractions (\cite{Lin}) but a negative answer for pairs of unitary matrices (\cite{dv:acfruowca}) or for triples of self-adjoint contractions (\cite{dv:sec, Dav2}).
When using the normalized Hilbert-Schmidt norm the question was formulated by Rosenthal \cite{Rosenthal}.
Here it has an affirmative answer for all finite sets of almost commuting unitaries, self-adjoint contractions or normal contractions (\cite{Glebsky}, \cite{FilonovKachkovskiy}, \cite{Hadwin-Li}).

There are two ways to formulate stability questions in the setting of operator algebras, either taking an $\epsilon$-$\delta$ approach a la Hyers and Ulam, or in terms of ``approximate'' $*$-homomorphisms  being ``close'' to actual $*$-homomorphisms.
(For precise definitions see Section \ref{Preliminaries}.)
When the norm under consideration is the operator norm, these questions are studied in the theory of (weakly) semiprojective $C^*$-algebras, which has found applications in Elliott's classification program for simple nuclear $C^*$-algebras and in the K-theory and E-theory for $C^*$-algebras.
For a study of stability for operator algebras with respect to trace norms see \cite{HadwinShulman}.

In the setting of group theory stability questions ask if ``approximate'' unitary representations of a given group are ``close'' to actual unitary representations. This topic recently got a lot of attention due to its connection with sofic, hyperlinear and operator-norm approximations of groups.
 Recall that a discrete group $G$ is {\it sofic} if there is an approximate representation $\phi_n \colon G \to Sym(n)$ to  permutations endowed with the Hamming distance which is separating in the sense that
$d_n(\phi_n(g), 1_{Sym(n)})$ is bounded away from
zero for all $ g\neq 1_G$.
A discrete group $G$ is {\it hyperlinear} ({\it MF} respectively) if there is an approximate representation $\phi_n \colon G \to \mathcal U(n)$ which is separating in the sense that
$\|\phi_n(g) -1_{n}\|$ is bounded away from
zero for all $ g\neq 1_G$. The norm here is the normalized Hilbert-Schmidt norm (the operator norm, respectively).
%A discrete group $G$ is {\it MF} if there is an approximate representation $\phi_n: G \to \mathcal U(n)$ which is separating in the sense that $\|\phi_n(g) -1_{Sym(n)}\|$ is bounded away from zero for all $ g\neq 1_G$. The norm here is the operator norm.
There are three fundamental problems concerning approximation properties of groups:
\begin{itemize}

\item {\it Sofic conjecture}:  Each discrete group is sofic.

%\medskip

\item {\it Connes Embedding Conjecture for groups}: Each discrete group is hyperlinear.

%\medskip

\item{\it MF-question}:  Is each discrete group MF?\footnote{Kirchberg \cite{bbek:gilfdc} conjectured that any stably finite C*-algebra is embeddable
into an norm-ultraproduct of matrix algebras, implying a positive answer to the
MF-question. Recent breakthrough
results imply that any amenable group is MF, see \cite{TWW}.}

\end{itemize}

%Recall that the general Connes Embedding Conjecture is that each $II_1$-factor embeds in the ultrapower of the hyperfinite $II_1$-factor.
%Its restriction to the von Neumann algebras

In \cite{AP} Arzhantseva and Paunescu observed that to disprove the sofic conjecture it is sufficient to find a non-residually finite (non-RF) group which is permutation stable. One also readily sees that to disprove the Connes Embedding Conjecture for groups (and hence the general Connes Embedding Conjecture) it is sufficient to find a non-RF finitely generated group which is stable with respect to the normalized Hilbert-Schmidt norm. And for a negative answer to the MF-question it is sufficient to find a non-RF finitely generated operator norm-stable group. Such an approach was recently successfully implemented by De Chiffre, Glebsky, Lubotzky and Thom for proving that not all groups are Frobenius-approximable (\cite{cglt:scvnog}). Their striking result was achieved by constructing a non-RF Frobenius norm stable group (here the Frobenius norm is the non-normalized Hilbert-Schmidt norm).

%arise in the context of sofic, hyperlinear and operator-norm approximations, see \cite{A} for a survey.

Stability of groups has become a question of great interest, and more and more new stability results are coming out.  For the stability with respect to the normalized Hilbert-Schmidt norm numerous examples and criteria can be found in
%there are various examples of stable and non-stable groups, see
\cite{FilonovKachkovskiy, Glebsky}, \cite{HadwinShulmanGroups}\footnote{In \cite{HadwinShulmanGroups} stability w.r.t.\ the normalized Hilbert-Schmidt norm was called matricial stability. However here we will call it Hilbert-Schmidt stability  while by matricial stability we always will mean stability w.r.t.\ the operator norm.}, \cite{1809.00632}.
There has been deep work done in the permutation stability setting (\cite{Glebsky2, AP, 1801.08381}, \cite{1811.00578}, \cite{1809.00632}, see also \cite{A} for a survey).
The most mysterious among group-theoretical versions of stability remains the operator norm stability. Although historically approximate representations of groups with respect to the operator norm appeared the first -- in the study of almost flat vector bundles on manifolds (\cite{CGM, ManMisch}) and in  relation to the K-theory of classifying spaces (\cite{ConnesHigson, MischMoh, Man}), almost nothing is known about the operator norm stability or non-stability of concrete groups.
In fact the only results mentioned in the literature is the non-stability of $\ZZ^2$, merely a reformulation of Voiculescu's result about two almost commuting unitaries, and the obvious stability of finitely generated free groups.

In this paper we initiate a systematic study of stability of discrete groups with respect to the operator norm.
We consider both stability for almost representations in all $C^*$-algebras and the more forgiving notion of matricial stability.
We obtain criteria for determining stability/non-stability of groups and produce an invariant which allows to detect when an approximate representation of a group is not close to an actual representation. The invariant has "topological" nature, is easily computable and applies to plenty of finitely presented groups.

With these tools in hand we explore stability of specific classes of groups.
%Our results broadly fall into two categories: Either we study a class of groups with powerful structure theory so that we can concretely describe the $C^*$-algebras and thus use $C^*$-algebraic tools to determine if they are stable, or we exploit concrete presentations of groups to provide finite dimensional counterexamples to stability.
We prove that virtually free groups are stable in the strongest possible sense.
We also investigate stability of virtually abelian groups and encounter a surprise:
Although $\mathbb Z^2$ is not matricially stable by Voiculescu's result, a virtually $\ZZ^2$ group may very well be!
Moreover, we prove that among 2-dimensional crystallographic groups (these are those which contain $\ZZ^2$ as a maximal abelian subgroup of finite index, and there are 17 of them) exactly 12 are matricially stable.
We use our results to determine stability/non-stability of all
crystallographic groups, all surface groups, and certain
Baumslag-Solitar groups.
We also show that a non-trivial finitely
generated torsion-free 2-step nilpotent group $G$ is matricially
stable if and only if $G \cong \ZZ$.

We finish by comparing operator norm stability with Hilbert-Schmidt stability.
All known results on the topic indicate that the Hilbert-Schmidt norm is more friendly to "small perturbations" than the operator norm.
We show here that for amenable groups this is indeed the case:  matricial stability with respect to the operator norm implies Hilbert-Schmidt stability.

\bigskip

\textbf{Acknowledgements.}
We gratefully acknowledge numerous helpful discussions in the early stages of this work with  Tim de Laat and Hannes Thiel concerning Section \ref{Virtually free groups} and with Dominic Enders concerning  Section \ref{Baumslag-Solitar groups}.
In  later stages we have  benefited substantially from discussions with %Marius Dadarlat,
 Jamie Gabe,  Sven Raum, Hannes Thiel, and Andreas Thom, and further thank the authors of \cite{KRTW} for making a version of their forthcoming paper available to us.

\SESTART We are grateful to the referees of this paper for useful comments. In particular, following an explicit suggestion to investigate a relation between our work and \cite{DadarlatNew} which we had previously not realized, we have opened a communication with Marius Dadarlat who has informed us that indeed some of the negative results presented below may be obtained by the methods developed there. Our methods are more direct but less general, so this interaction has opened up for a more systematic approach to some questions we leave open here. Dadarlat informs us that he will elaborate these ideas in a forthcoming note \cite{DadarlatComing} to which we will refer in footnotes when relevant.\SEEND

% concerning Section \ref{Baumslag-Solitar groups} and  with Tim de Laat and Hannes Thiel concerning Section \ref{Virtually free groups}.
%We also would like to thank Jamie Gabe for useful communication about quasidiagonal traces.

The first named author was supported by the DFF-Research Project 2 'Automorphisms and Invariants of Operator Algebras', no. 7014-00145B, and by the Danish National Research Foundation through the Centre for Symmetry and Deformation (DNRF92).
The second-named author was supported by the Polish National Science Centre grant under the contract number 2019/34/E/ST1/00178.
Finally, the research of the first- and second-named author was supported by the grant H2020-MSCA-RISE-2015-691246-QUANTUM DYNAMICS.

\section{Preliminaries}\label{Preliminaries}

\subsection{Stability (semiprojectivity) for general $C^*$-algebras}

We begin by recalling the definition of (matricially) (weakly) semiprojective $C^*$-algebras.

\begin{definition}[{\cite{bb:stc}}]\label{spdef}
A separable $C^*$-algebra $A$ is \emph{semiprojective} if for every separable $C^*$-algebra $B$, every increasing sequence of ideals $J_1 \subseteq J_2 \subseteq \cdots$ in $B$, and every $*$-homomorphism $\phi \colon A \to B/\overline{\bigcup_k J_k}$, there exists an $n \in \NN$ and a $*$-homomorphism $\psi \colon A \to B/J_n$ such that
\[
	\pi_{n,\infty} \circ \psi = \phi,
\]
where $\pi_{n,\infty} \colon B/J_n \to B/\overline{\bigcup_k J_k}$ is the natural quotient map.
\end{definition}

The following diagram shows the lifting problem one has to solve to prove that a $C^*$-algebra $A$ is semiprojective.
The $*$-homomorphisms associated with the solid arrows are given, and the task is to find $n \in \NN$ and a $*$-homomorphism that fits on the dashed arrow so that the diagram commutes.

\[
	\xymatrix{
		& B/J_n \ar@{->>}[d]^{\pi_{n,\infty}} \\
		A \ar[r]_-{\phi} \ar@{-->}[ur]^-{\psi} & B/\overline{\cup_n J_n}.	
	}
\]

\noindent We may assume without loss of generality that the map $\phi$ is either injective, surjective, or both (\cite[Proposition 2.2]{bb:ssc}).

For finitely presented $C^*$-algebras we can translate the notion of semiprojectivity into an $\epsilon$-$\delta$ version of stability.
There are various definitions of $C^*$-algebra relations (\cite{bb:stc, LoringBook, tal:car}), but for this paper it suffices to consider $*$-polynomials in finitely many (non-commuting) variables and the restrictions on the norms of generators.

\begin{definition}[{\cite{LoringBook}}]
A finitely presented $C^*$-algebra
%\[  A = C^* \left( \genrel{x_1, x_2 \ldots, x_n}{\mathcal R_i(x_1, x_2, \ldots, x_n)=0, i = 1, 2, \ldots, N} \right), \]
\begin{multline*}
 A = C^* \langle x_1, x_2 \ldots, x_n \;|\; \mathcal R_i(x_1, x_2, \ldots, x_n)=0,  i = 1, 2, \ldots, N,\\
  \|x_k\| \le C_k, k=1, \ldots, n \rangle
\end{multline*}
is \emph{stable} if, for every $\epsilon > 0$ there is a $\delta>0$ such that:
Given a surjection $\pi \colon D \to B$ and $z_1, z_2, \ldots, z_n \in D$ for which $\|\mathcal R_i (z_1, z_2, \ldots, z_n)\|<\delta$, $i=1,2, \ldots, N$, and for which $\rho \colon A \to B$ given by $\rho(x_k) = \pi(z_k)$ is a $*$-homomorphism, there exists a $*$-homomorphism $\phi \colon A \to D$ such that $\|\phi(x_k) - z_k\|\le \epsilon$ and $\pi(z_k) = \pi(\phi(x_k)),$ $k=1, 2, \ldots, n$.
\end{definition}

By \cite[14.1.4]{LoringBook} the notions of stability and semiprojectivity coincide for finitely presented $C^*$-algebras.

We weaken the notion of semiprojectivity in two steps.

\begin{definition}[{cf. \cite{setal:ccsr}}]\label{weaklydef}
A separable $C^*$-algebra $A$ is \emph{weakly semiprojective} if for every sequence of separable $C^*$-algebras $B_1, B_2, \ldots$, and every $*$-homomorphism $\phi \colon A \to \prod_n B_n / \bigoplus_n B_n$, there exists a $*$-homomorphism $\psi \colon A \to \prod_n B_n$ such that
\[
	\rho \circ \psi = \phi,
\]
where $\rho \colon \prod_n B_n \to \prod_n B_n / \bigoplus_n B_n$ is the natural quotient map.
\end{definition}

\begin{definition}[{cf. \cite{setal:ccsr}}]\label{matrixdef}
A separable $C^*$-algebra $A$ is \emph{matricially weakly semiprojective} if for every sequence of matrix algebras $\matrM{k_1}, \matrM{k_2}, \ldots$, and every $*$-homomorphism $\phi \colon A \to \prod_n \matrM{k_n} / \bigoplus_n \matrM{k_n}$, there exists a $*$-homomorphism $\psi \colon A \to \prod_n \matrM{k_n}$ such that
\[
	\rho \circ \psi = \phi,
\]
where $\rho \colon \prod_n \matrM{k_n} \to \prod_n \matrM{k_n} / \bigoplus_n \matrM{k_n}$ is the natural quotient map.
\end{definition}

One can easily show that the matrix algebras in the definition above can be replaced by arbitrary finite-dimensional $C^*$-algebras.

\begin{remark}
Let $A$ be a $C^*$-algebra.
If $A$ is semiprojective then $A$ is also weakly semiprojective since we can `pad' with zeros to lift from $\prod_n B_n / \bigoplus_{n=1}^{n_0} B_n$ all the way to $\prod_n B_n$.
Clearly if $A$ is weakly semiprojective then $A$ is also matricially weakly semiprojective.
\end{remark}

There are also $\epsilon$-$\delta$ stability properties corresponding to the weakenings of semiprojectivity.

\begin{definition}[\cite{setal:ccsr}, \cite{LoringBook}]
A finitely presented $C^*$-algebra
\begin{multline*}
 A = C^* \langle x_1, x_2 \ldots, x_n \;|\; \mathcal R_i(x_1, x_2, \ldots, x_n)=0,  i = 1, 2, \ldots, N,\\
  \|x_k\| \le C_k, k=1, \ldots, n \rangle
\end{multline*}
is \emph{weakly stable} if, for every $\epsilon > 0$ there is a $\delta>0$ such that:
Given a $C^*$-algebra $B$ and $z_1, z_2, \ldots, z_n\in B$ for which $\|\mathcal R_i (z_1, z_2, \ldots, z_n)\|<\delta$, $i=1, 2, \ldots, N$, there exists a $*$-homomorphism $\phi \colon A \to B$ such that $\|\phi(x_k) - z_k\|\le \epsilon$,$ k=1, 2, \ldots, n$.
\end{definition}

\begin{definition}[\cite{elp:sar}]
A finitely presented $C^*$-algebra
\begin{multline*}
 A = C^* \langle x_1, x_2 \ldots, x_n \;|\; \mathcal R_i(x_1, x_2, \ldots, x_n)=0,  i = 1, 2, \ldots, N,\\
  \|x_k\| \le C_k, k=1, \ldots, n \rangle
\end{multline*}
is \emph{matricially stable} if, for every $\epsilon > 0$ there is a $\delta>0$ such that:
Given a matrix $C^*$-algebra $B$ and $z_1, z_2, \ldots, z_n \in B$ for which $\|\mathcal R_i (z_1,z_2,  \ldots, z_n)\|<\delta$, $i=1,2, \ldots, N$, there exists a $*$-homomorphism $\phi \colon A \to B$ such that $\|\phi(x_k) - z_k\|\le \epsilon$, $k=1,2, \ldots, n$.
\end{definition}

Weak stability is probably the most natural property of stability for $C^*$-algebras as it requires ``almost'' $*$-homomorphisms to be ``close'' to $*$-homomorphisms.
For finitely presented $C^*$-algebras weak semiprojectivity coincides with the weak stability property (\cite[Theorem 4.6]{setal:ccsr}) and matricial weak semiprojectivity coincides with matricial stability (\cite[Theorem 5.6]{setal:ccsr}).

We note that not all $C^*$-algebras admit $*$-homomorphisms into algebras of the form $\prod_n \matrM{k_n} / \bigoplus_n \matrM{k_n}$.
If a $A$ has an embedding into $\prod_n \matrM{k_n} / \bigoplus_n \matrM{k_n}$ for some sequence $(k_n)$ it is called MF (see \cite{bbek:gilfdc}).
A $C^*$-algebra is called residually finite-dimensional (RFD) if it has a separating family of finite-dimensional representations.

\begin{proposition}[{see \cite[Proposition 13.4]{nbp:q}}] \label{RFD}
If $A$ is MF and matricially weakly semiprojective, then $A$ is RFD.
\end{proposition}
\begin{proof}
Pick an embedding $\phi \colon A \to \prod_n \matrM{k_n} / \bigoplus_n \matrM{k_n}$.
Since $A$ is matricially weakly semiprojective $\phi$ lifts to a map $\psi \colon A \to \prod_n \matrM{k_n}$ such that $\phi = \rho \circ \psi$.
As $\phi$ is injective, we must have that $\psi$ is also injective, and therefore $A$ is RFD.
\end{proof}

\begin{corollary}
If $A$ is MF, simple, and infinite-dimensional, then $A$ cannot be matricially weakly semiprojective.
\end{corollary}

\subsection{Stability for groups}

The focus of this paper is to study when group $C^*$-algebras are ((matricially) weakly) semiprojective.
We will mainly be concerned with discrete countable groups $G$, in which case $C^*(G)$ is unital and separable.
This allows us to disregard the complications regarding weak semiprojectivity for non-unital $C^*$ presented in \cite{tal:pwsac}.
Nevertheless, we give the main definitions as generally as possible.

\begin{definition}
Let $G$ be a second countable and locally compact topological group and let $C^*(G)$ denote the full group $C^*$-algebra of $G$.
\begin{enumerate}[(i)]
\item We say $G$ is \emph{\SP} when $C^*(G)$ is semiprojective.
\item We say $G$ is \emph{\WSP} when $C^*(G)$ is weakly semiprojective.
\item We say $G$ is \emph{\MWSP} when $C^*(G)$ is matricially semiprojective.
\end{enumerate}
\end{definition}

\begin{remark}
As implied by the names we have that if $G$ is \SP\ then $G$ is \WSP, and if $G$ is \WSP\ then $G$ is \MWSP.
\end{remark}

Weak $C^*$-stability and matricial stability can be reformulated in terms of approximate homomorphisms for finitely presented discrete groups, see Propositions \ref{prop:wspandalmostreps} and \ref{prop:mwspandalmostreps}.
For a unital $C^*$-algebra $A$ its unitary group will be denoted by $\mathcal U(A)$.

\begin{definition}
Let $G$ be a discrete group and let $A_n$, $n\in \mathbb N$, be a sequence of $C^*$-algebras.
A sequence of maps $\phi_n \colon G \to \mathcal U(\mathcal A_n)$ is an \emph{approximate homomorphism} if
\[
	\lim_{n\to \infty} \|\phi_n(g_1g_2) - \phi_n(g_1)\phi_n(g_2)\| = 0, \quad \text{ for all } g_1, g_2 \in G.
\]
\end{definition}

When all the $A_n$ are finite-dimensional $C^*$-algebras we refer to approximate homomorphisms as finite-dimensional approximate representations.
The following proposition is trivial.

\begin{proposition}
Let $G$ be a discrete group.
\begin{enumerate}
 \item $G$ is weakly $C^*$-stable if for any approximate homomorphism  $\phi_i \colon G \to \mathcal U(A_i)$ there are $*$-homomorphisms $\rho_i \colon C^*(G) \to A_i$ such that $\lim_{i\to \infty} \|\phi_i(g) - \rho_i(g)\|=0$, for all $g\in G$.
 \item $G$ is matricially stable if for any finite-dimensional approximate representation $\phi_i \colon G \to \mathcal U(A_i)$ there are $*$-homomorphisms $\rho_i \colon C^*(G) \to A_i$ such that $\lim_{i\to \infty} \|\phi_i(g) - \rho_i(g)\|=0$, for all $g\in G$.
\end{enumerate}
\end{proposition}

For a finitely presented group $G$ the definitions above can be given an $\epsilon$-$\delta$ formulation that only reference the unitary representations of $G$.

\begin{definition}
Let $G$ be a finitely presented discrete group, say
\[
	G = \genrel{S}{R} = \genrel{ g_1, g_2, \ldots, g_s  }{ r_{1}, r_2, \ldots, r_{l} }.
\]
Assume further that the generating set $S$ of $G$ is symmetric, i.e. if $g \in S$ then $g^{-1} \in S$.
Let $A$ be a $C^*$-algebra, and let $\epsilon > 0$.
A function $f \colon S \to \mathcal{U}(A)$ is an \emph{$\epsilon$-almost homomorphism} if
\[
	\|1 - r_{j}\left( f(g_{1}), f(g_{2}), \ldots, f(g_{s})\right) \| \le \epsilon
\]
for all $j = 1, 2, \ldots, l$.
\end{definition}

The proofs of the following propositions are straight forward once one realizes two things.
One, that a discrete group embeds into the unitary group of its C*-algebra $C^*(G)$. Moreover if a group $G$  has a presentation $$G = \genrel{S}{R} = \genrel{ g_1, g_2, \ldots, g_s  }{ r_{1}, r_2, \ldots, r_{l} },$$ then $C^*(G)$ has the presentation $$C^*(G)  = \genrel{ u_1,  \ldots, u_s  }{ u_i^*u_i=1=u_iu_i^*, \; r_{j}(u_1, \ldots, u_s) = 1, i=1, \ldots, s, j= 1, \ldots, l.}$$
Two, there is a folklore fact that an element of a C*-algebra which is almost unitary (that is  almost satisfies the relation $u^*u=1=uu^*$) must be close to some unitary element.
So if we are given an almost representation of the group $C^*$-algebra, we can move the images of generators a little to make sure they are unitaries, while making sure they still almost satisfy the group relations.

\begin{proposition} \label{prop:wspandalmostreps}
Let $G$ be a finitely presented discrete group.
Them G is weakly $C^*$-stable if and only if the following holds:
For any $\epsilon> 0$ there is $\delta > 0$ such that for any unital $C^{*}$-algebra $A$ and for any $\delta$-almost homomorphism $f \colon S \to \mathcal{U}(A)$ there is a homomorphism $\pi \colon G \to \mathcal{U}(A)$ such that
\[
	\| \pi(g) - f(g)\| \le \epsilon, \quad \text{ for all } g \in S.
\]
\end{proposition}

\begin{proposition} \label{prop:mwspandalmostreps}
Let $G$ be a finitely presented discrete group.
Then $G$ is matricially stable if and only if the following holds:
For any $\epsilon > 0$ there is $\delta > 0$ such that for any  finite-dimensional $C^*$-algebra $A$ and for any $\delta$-almost homomorphism $f \colon S \to \mathcal{U}(A)$ there is a homomorphism $\pi \colon G \to \mathcal{U}(A)$ such that
\[
	\| \pi(g) - f(g)\| \le \epsilon, \quad \text { for all } g \in S.
\]
\end{proposition}

Some classical results about (matricially) (weakly) semiprojective $C^*$-algebras can be easily translated into the group setting to give the first examples of \SP\ groups.
These results  are all well known and we will reprove them several times in what follows below as special cases.

\begin{example} \
\begin{enumerate}[(i)]
 \item Finite groups are all \SP: If $G$ is a finite group, then $C^*(G)$ is finite dimensional and therefore semiprojective by \cite[Corollary 2.30]{bb:stc}.
 \item The integers form a \SP\ group: It is well known that $C^*(\ZZ) \cong C(\TT)$ which is semiprojective since unitaries lift as required.
 \item The finitely generated free groups $\FF_n$ are \SP\ since $C^*(\FF_n)$ is the universal $C^*$-algebra generated by $n$ unitaries which may be lifted individually.
 \item The infinite dihedral group
 \[
 \ZZ_2 \star \ZZ_2= \genrel{ r,t }{ r^2=t^2=e }
 \]
  is semiprojective since $C^*(\ZZ_2 \star \ZZ_2)$ is the universal $C^*$-algebra generated by two symmetries (equivalently, by two projections) and it is straightforward to lift these individually.
\end{enumerate}
\end{example}

We also mention three examples which demonstrate that the three classes are different.

\begin{example} \
\begin{enumerate}
 \item The infinitely generated free group $\FF_\infty$ is not \SP, as it fails to satisfy the conclusion of \cite[Corollary 2.10]{bb:ssc}, but it is \WSP. Indeed, one may just lift each unitary generator separately to $\prod B_n$ as in Definition \ref{weaklydef} and use the universal property to create a lifting.
 \item $\ZZ^2$ is not \MWSP\ since $C^*(\ZZ^2)=C(\TT^2)$ is not matricially semiprojective. Indeed, the famous Voiculescu matrices  (\cite{dv:acfruowca}, see also \cite{el:acum}) demonstrate that there exist sequences of almost commuting unitary matrices that are never close to exactly commuting unitary matrices.
 \item The group $C^*$-algebra associated to the crystallographic group
 \[
 \crystal{pg}= \genrel{ x,y }{ xy=y^{-1}x }= \genrel{ p,q }{ p^2 = q^2 }.
 \]
(the two representations are isomorphic via $p=x,q=xy$) was proved to be \MWSP\ in \cite{elp:sar}. It is not \SP\ by the work of Enders (\cite{ed:cssc}), and we prove below that it is not \WSP\ either. This group may also be described as the fundamental group of the Klein bottle or as the Baumslag-Solitar group $\baso(1,-1)$.
 \end{enumerate}
\end{example}

Recall that a group is {\it maximally almost periodic} (MAP) if it has a separating family of finite-dimensional representations.
A group is {\it residually finite} (RF) if it has a separating family of homomorphisms to finite groups.

A group is MF if it embeds into the unitary group of  $\prod_n \matrM{k_n} / \bigoplus_n \matrM{k_n}$.
If a group $G$ does not admit any homomorphism to the unitary group of $\prod_n \matrM{k_n} / \bigoplus_n \matrM{k_n}$, then $G$ is automatically matricially stable.
However if such a group exists it must be rather exotic, since by Kirchberg's conjecture all discrete groups are MF.
All discrete amenable groups are known to be MF by \cite{TWW}.

\begin{proposition} \label{prop:MFisNice}
Let $G$ be a discrete  group.
The following holds:
$$ G \;\text{is MF and matricially stable}  \Rightarrow G \;\text{is}\; MAP.$$
If $G$ is finitely generated,  then
$$ G \;\text{is MAP} \Leftrightarrow G \;\text{is RF}.$$
\end{proposition}
\begin{proof}
The first statement is obvious.
The second statement is folklore. We write its proof for the reader's convenience.
Suppose $G$ is MAP and let $e\neq g\in G$. Then there is a finite-dimensional representation $\pi$ of $G$ such that $\pi(g)\neq 1$. Since $\pi(G)$ is a finitely generated linear group, by the Malcev Theorem it is RF (\cite{am:imrifg}, Th. VII). Therefore there exists a homomorphism $f$ from $\pi(G)$ to a finite group such that $f(\pi(g))\neq e$. Thus the homomorphism $f\circ \pi$ to a finite group is not trivial on $g$. This means $G$ is RF.
\end{proof}

\begin{corollary}
If $G$ is discrete, amenable, and matricially stable then $G$ is MAP.
\end{corollary}

\subsection{The Voiculescu matrices} \label{sec:voiculescu}

Throughout this paper we will often rely on a sequence of matrices introduced by Voiculescu:
For each $n \in \NN$ we let $\omega_n = \exp(\frac{2 \pi i}{n})$ and define $S_n, \Omega_n \in \matrM{n}$ by
\[
	S_n =
	\begin{pmatrix}
	0 & 0 & \cdots & 0 & 1 \\
	1 & 0 & \cdots & 0 & 0 \\
	0 & 1 & \cdots & 0 & 0 \\
	  &   & \ddots & & \\
	0 & 0 & \cdots & 1 & 0
	\end{pmatrix}	
	\quad
	\text{ and }
	\quad
	\Omega_n =
	\begin{pmatrix}
	\omega_n^1 & 0 & \cdots & 0 \\
	0 & \omega_n^2 & \cdots & 0 \\
	  & & \ddots & \\
	0 & 0 & \cdots & \omega_n^n
	\end{pmatrix}.
\]
We call these the \emph{Voiculescu matrices}.
They are the prime example of almost commuting unitaries that are not close to exactly commuting unitaries, as shown by Voiculescu in \cite{dv:acfruowca}.
Rather than following the ideas of \cite{dv:acfruowca}, we will extend the winding number approach from \cite{el:acum}.
We will explain in Section \ref{sec:ExelLoring} how to associate a winding number invariant to a set of unitaries, and throughout Section \ref{sec:classes} we will then use the Voiculescu matrices to construct paths with non-zero winding number.

We now fix our notation for commutators.
If $A,B$ either are elements of a group or invertible elements in a $C^*$-algebra then $\mcomm{A}{B}$ denotes the multiplicative commutator of $A$ and $B$, that is
\[
\mcomm{A}{B} = ABA^{-1}B^{-1}.
\]
For any two $A,B$ in a $C^*$-algebra we use $\acomm{A}{B}$ to denote their additive commutator, i.e.,
\[
\acomm{A}{B} = AB - BA.
\]
We note that if $U,V$ are unitaries in some unital $C^*$-algebra then
\[
\|\mcomm{U}{V} - 1\| = \|\acomm{U}{V}\|.
\]

\section{Stability criteria}

We present here results which may be used to decide stability properties.
We will show that any virtually free group is \SP, and provide useful criteria for determining that certain groups fail to have the weaker properties.

\subsection{Virtually free groups}\label{Virtually free groups}

Recall that a group is called virtually free if it contains a free subgroup of finite index.
If $H$ is a free subgroup of $G$ of finite index, we can define a normal subgroup $K$ of $G$ by
\[
	K = \bigcap_{g \in G} g H g^{-1}.
\]
Note that the intersection is only over finitely many different conjugates, since $H$ has finite index, and therefore $K$ must also have finite index.
Further, as $K$ is a subgroup of $H$ it is free by the Nielsen-Schreier Theorem (see for instance \cite[Theorem 1A.4]{ah:at}.
Thus, any virtually free group contains a free normal subgroup of finite index.

Our proof that all finitely generated virtually free groups are $C^*$-stable is built on a description of such groups as HNN extensions of tree products in \cite{akapds:ficefg}, and we thank Tim de Laat and Hannes Thiel for drawing this method to our attention.

We first make some general observations about stability, amalgamated products and HNN extensions.

\begin{lemma} \label{lem:AmalgamtedproductsCommute}
If $\Gamma = G_1 \star_H G_2$ is an amalgamated product of discrete countable groups then $C^*(\Gamma) \cong C^*(G_1) \star_{C^*(H)} C^*(G_2)$.
\end{lemma}
\begin{proof}
We have the following commutative diagram
\[
\xymatrix{
	C^*(H) \ar[r] \ar[d]		& C^*(G_1) \ar[d]^{\iota_1} \\
	C^*(G_2) \ar[r]_{\iota_2}	& C^*(\Gamma)
}
\]
where all the maps are induced by the corresponding group maps.
Since $\Gamma$ is generated by the copies inside it of $G_1$ and $G_2$, we have that $\iota_1(C^*(G_1)) \cup \iota_2(C^*(G_2))$ generate $C^*(\Gamma)$.
Hence to show $C^*(\Gamma) \cong C^*(G_1) \star_{C^*(H)} C^*(G_2)$ we only need to show, that $C^*(\Gamma)$ has the right universal property.
So suppose we have $*$-homomorphisms $\phi_i \colon C^*(G_i) \to A$, $i=1,2$, that agree on $C^*(H)$.
By replacing $A$ by $\phi(1)A\phi(1)$, we can assume that $A$ is unital.
Hence $\phi_i$ yields a representation of $G_i$ and furthermore these representation agree on $H$.
The universal property of $\Gamma$ now gives a representation of $\Gamma$ in the unitary group of $A$, which in turn gives us a $*$-homomorphism $\psi \colon C^*(\Gamma) \to A$.
We have that
\[
	\phi_i(u_g) = \psi(u_g), \quad g \in G_i,
\]
for $i = 1,2$.
\end{proof}

\begin{corollary} \label{cor:AmalgamatedProductSp}
Let $\Gamma = G_1 \star_H G_2$ be an amalgamated product of discrete countable groups.
If $G_i$ is $C^*$-stable for $i=1,2$ and $H$ is finite, then $\Gamma$ is $C^*$-stable.
\end{corollary}
\begin{proof}
By \cite[Proposition 2.32]{bb:stc} amalgamated products of semiprojective $C^*$-algebras over finite dimensional $C^*$-algebras are semiprojective.
\end{proof}

\begin{definition}[see \cite{rclpes:cgt}]
Let $G$ be a group with presentation $G = \genrel{ S }{ R }$, let $H,K$ be isomorphic subgroups of $G$, and let $\alpha \colon H \to K$ be an isomorphism. Let $t$ be a new symbol not in $S$, and define
\[
    G \star_{\alpha} = \genrel{ S,t }{ R, tht^{-1}=\alpha(h), \text{ for all } h \in H }.
\]
The group $G \star_{\alpha}$ is called the HNN extension of $G$ relative to $H,K$, and $\alpha$.
\end{definition}

In \cite{yu:rheoa} two notions of HNN extensions for $C^*$-algebras are discussed.
As one might expect, if we have a unital $C^*$-algebra $A$, with isomorphic unital (same unit as $A$) subalgebras $B,D$, and isomorphism $\alpha \colon B \to D$, then the universal HNN extension of $A$ by $\alpha$ is the $C^*$-algebra generated by $A$ and a unitary $u$ that implements $\alpha$ on $B$.
We denote this by $A \star_{\alpha}$.
Ueda also considers a reduced HNN extension.
Similarly to the case for amalgamated products HNN extensions of discrete countable groups ``commutes'' with taking $C^*$-algebras of groups.

\begin{lemma} \label{lem:HNNCommute}
Let $G$ be a countable discrete group, $H,K$ isomorphic subgroups of $G$ with isomorphism $\alpha$.
Then $C^*(G \star_{\alpha}) \cong C^*(G) \star_{\tilde{\alpha}}$, where $\tilde{\alpha}$ denotes the isomorphism between the unital subalgebras of $C^*(G)$ generated by $H$ and $K$.
\end{lemma}
\begin{proof}
Similar to the proof of Lemma \ref{lem:AmalgamtedproductsCommute}.
\end{proof}

\begin{proposition} \label{prop:HNNisSP}
Let $G$ be a $C^*$-stable countable discrete group and let $\alpha \colon H \to K$ be an isomorphism between two subgroups of $G$.
If $H$ (and hence $K$) is finite, then $G \star_{\alpha}$ is $C^*$-stable.
\end{proposition}
\begin{proof}
For ease of notation put $A = C^*(G \star_{\alpha})$ and $D = C^*(H)$.
Following \cite{yu:rheoa} we let $\cA = (C^*(G) \otimes \mathbf{M}_2) \star_{D \oplus D} (D \otimes \mathbf{M}_2)$.
By Lemma \ref{lem:HNNCommute} and \cite[Proposition 3.3]{yu:rheoa} we have $\mathbf{M}_2(A) \cong \cA$.
Hence $A$ is semiprojective if $\cA$ is (\cite[Proposition 2.27]{bb:stc}).

Since $H$ is finite $D$ is a finite-dimensional $C^*$-algebra. Hence to see that $\cA$ is semiprojective, it suffices to show that $C^*(G) \otimes \matrM{2}$ and $D \otimes \matrM{2}$ are semiprojective (and then apply \cite[Proposition 2.32]{bb:stc}).
By assumption $C^*(G)$ is semiprojective and $D$ is semiprojective since it is finite dimensional.
Therefore $C^*(G) \otimes \matrM{2}$ and $D \otimes \matrM{2}$ are semiprojective.
\end{proof}

The final group construction we need is tree products.
For the definition we refer to section 2 of \cite{gkcyt:spctpg}.

\begin{lemma} \label{lem:TreeProduct}
Let $G$ be a finite tree product with finite edge groups.
If the vertex group $G_v$ is $C^*$-stable for all vertices $v$, then $G$ is $C^*$-stable.
\end{lemma}
\begin{proof}
We will prove the lemma by induction on the number of vertices.
If there is only one vertex, $v$ say, there are no edges, and so the tree product is simply the vertex group $G_v$.
By assumption $G_v$ is $C^*$-stable.

Assume now, that the Lemma is true for all trees with $n$ vertices, and that we are given a tree product with $n+1$ vertices.
By \cite[Remark 3.5]{gkcyt:spctpg} we can write the tree product as $G \cong A \star_{C} T$, where $A$ is the vertex group of a leaf, $C$ is the edge group connecting $A$ to the tree, and $T$ is the tree product of the other $n$ vertex groups.
The inductive hypothesis tells us that $T$ is $C^*$-stable, and by assumption $A$ is $C^*$-stable.
Since $C$ is finite, it follows from Corollary \ref{cor:AmalgamatedProductSp} that $G$ is $C^*$-stable.
\end{proof}

We now have all we need to show that finitely generated virtually free groups are $C^*$-stable.

\begin{theorem}\label{vfissp}
All finitely generated virtually free groups are \SP.
\end{theorem}
\begin{proof}
Let $G$ be a finitely generated virtually free group.
By \cite[Theorem 1]{akapds:ficefg} there is a finite tree product of finite groups $K$, such that $G$ is an iterated HNN extension of $K$ by finite groups.
Lemma \ref{lem:TreeProduct} tells us that $K$ is \SP.
Repeated applications of Proposition \ref{prop:HNNisSP} then gives that $G$ is \SP.
\end{proof}

\subsection{Exel-Loring type invariant} \label{sec:ExelLoring}

Here we introduce an invariant which allows to see that certain groups fail to be matricially stable.
It will be applied in Sections \ref{Crystal groups}--\ref{Baumslag-Solitar groups}

\begin{definition}
A relation
\[
	R(x_1,\ldots, x_N) = x_{i_1}^{k_1}\ldots x_{i_s}^{k_s}
\]
is \emph{homogeneous} if
\begin{equation}\label{form}
	\sum_{\{j:  i_j = i\}} k_j  = 0,
\end{equation}
for each $1 \le i \le N$.
For any homogeneous relation $R$, we let $L(R)=\sum_j |k_j|$.
\end{definition}

We will consider groups $G = \genrel{ x_1, \ldots, x_N }{ R_l(x_1,\ldots, x_N) = e, l \in \mathbb N }$ presented by relations some of which are homogeneous, and we will see that for them there is a ``winding number obstruction'' for being \MWSP, very similar to the Exel-Loring winding number obstruction (\cite{el:acum}) for the commutation relation.

Let $R$ be some group relation.
For any invertible matrices $V_1, \ldots, V_N $ define a curve $\gamma^{V_1,\ldots, V_N}$ by
\[
	\gamma^{ V_1,\ldots, V_N}(r)  = \det (r R(V_1,\ldots, V_N) + (1-r)\myu), 0\le r\le 1.
\]
When $R$ is homogeneous, $\det R(V_1,\ldots, V_N) = 1$ and hence
\[
	\gamma^{V_1,\ldots, V_N}(0) = 1 = \gamma^{V_1,\ldots, V_N}(1).
\]
Thus $\gamma^{V_1,\ldots, V_N}$ is a closed curve.

\begin{theorem}\label{exelloring}
Let $R$ be a homogeneous relation, and let $X_1, X_2, \ldots, X_N \in \matrM{n}$ be unitary matrices. Suppose that
\[
	\| R(X_1,X_2,\ldots, X_N) - \myu_n \|<1/2.
\]
If there exist unitary matrices $A_1, A_2, \ldots, A_N\in \matrM{n}$ such that $R(A_1,A_2, \ldots, A_N) = \myu_n$ and $ \max_{1\le i\le N} \|X_i-A_i\| < \frac{1}{2L(R)}$,  then
\[
	\wind \gamma^{X_1,X_2,\ldots, X_N} = 0.
\]
\end{theorem}
\begin{proof}
For each $0\le t\le 1$ let
\begin{equation} \label{eq:ait}
	A_i^{(t)} = tA_i + (1-t)X_i = t(A_i-X_i) + X_i.
\end{equation}
Then
\begin{equation}\label{eq:ait-x}
	\|A_i^{(t)}-X_i\|\le \|A_i-X_i\|< \frac{1}{2L(R)}.
\end{equation}
In particular it follows that all $A_i^{(t)}$ are invertible.
Then for each $0\le t\le 1$ the curve $\gamma^{A_1^{(t)},\ldots, A_N^{(t)}}$ is well-defined and it is a homotopy between $\gamma^{X_1,\ldots, X_N}$ and $\gamma^{A_1,\ldots, A_N}$.
We will show that for any $0\le t\le 1 $, $\gamma^{A_1^{(t)},\ldots, A_N^{(t)}}$ does not go through $0$.
Using (\ref{eq:ait-x}), we obtain by standard estimates that 
$$\left\|\left(A_i^{(t)}\right)^{-1}\right\|\le \frac{1}{1-\|A_i^{(t)} - X_i\|} < 2.$$ Therefore we obtain
\begin{equation}\label{eq:ait-xinv}
\|(A_i^{(t)})^{-1} - (X_i)^{-1}\|\le \|(A_i^{(t)})^{-1}\| \|A_i^{(t)}-X_i\| \|(X_i)^{-1}\| < \frac{1}{L}.\end{equation}
So
\begin{gather*}
 \|(r R(A_1^{(t)},\ldots, A_N^{(t)}) + (1-r)\myu) - R( X_1,\ldots, X_N) \|  \\\le   r\|R( A_1^{(t)},\ldots, A_N^{(t)}) - R( X_1,\ldots, X_N)\| + (1-r)\|\myu_n - R( X_1,\ldots, X_N)\|\\
  <
 r L \frac{1}{L} + (1-r) \frac{1}{2} \le  1
\end{gather*}
(here the summand $L \frac{1}{L}$ was obtained by applying the standard adding-subtracting trick $L$ times and using (\ref{eq:ait-x}) and (\ref{eq:ait-xinv})).
Hence $r R(A_1^{(t)},\ldots, A_N^{(t)}) + (1-r)\myu_n $ is at distance less than 1 from the unitary matrix $R( X_1,\ldots, X_N)$ and hence is invertible itself.
Hence its determinant is not zero, which means that $\gamma^{A_1^{(t)},\ldots, A_N^{(t)}}$ does not go through $0$.
Thus for each curve in the homotopy, the winding number is well-defined.
Since $R(A_1, \ldots, A_N)=\myu_n$, the curve $\gamma^{A_1, \ldots, A_N}$ is a constant curve, so $\wind \gamma^{A_1, \ldots, A_N} = 0.$
Since the winding number is a homotopy invariant, we conclude that
\[
	\wind \gamma^{X_1, \ldots, X_N} = \wind \gamma^{A_1, \ldots, A_N} = 0. \qedhere
\]
\end{proof}

\subsection{Virtually abelian groups}

We recall that a group $G$ is said to be virtually abelian if it contains an abelian subgroup $H$ of finite index.
Our interest here is finitely generated virtually abelian groups. Since a finite index subgroup of finitely generated group is necessarily  finitely generated,  a finitely generated virtually abelian group contains a finitely generated abelian subgroup $H$ of finite index, and by the structure of finitely generated abelian groups we can further assume that $H \cong \ZZ^m$ for some $m \in \NN$.
Arguing as for the virtually free groups, we can arrange for $H$ to be normal.
Thus a finitely generated group $G$ is virtually abelian if and only if it contains a normal subgroup $H$ isomorphic to $\ZZ^m$ for some $m \in \NN$.
The number $m$ is unique, and we call it the rank of $G$.

Let $G$ be a finitely generated virtually abelian group.
If $G$ has rank $1$, then it is virtually free and therefore \SP\ by Theorem \ref{vfissp}.
To understand the case where $G$ has rank $m > 1$, we first build a $*$-homomorphism $\alpha \colon C^*(G) \to C(\II^m, \matrM{N})$.
In the case $m = 2$ we use $\alpha$ to produce an unsolvable lifting problem that proves $G$ is not \WSP.
If $m \geq 3$ we combine $\alpha$ with the Voiculescu matrices to build an unsolvable lifting problem showing that $G$ is not even \MWSP.

\begin{theorem} \label{MapToScalarFunctions}
Let $G$ be a group and $H\cong \mathbb Z^m$ be a normal subgroup of index $N < \infty$.
There exists a $*$-homomorphism $\alpha \colon C^*(G) \to C(\II^m, \matrM{N})$ such that $\alpha(C^*(H))$ contains all the scalar-valued functions, i.e.,
\[
	\alpha(C^*(H))\supseteq C(\II^m, \mathbb C \myu_N).
\]
\end{theorem}
\begin{proof}
If $G$ is abelian, then it is the product of $H$ and a cyclic group of cardinality $N$, so $C^*(G) = C(\mathbb T^m) \otimes \mathbb C^N$ and the statement follows trivially.
So below we assume that $G$ is non-abelian.

Since $H \cong \ZZ^m$, we can denote its elements as $\vec n = (n_1, n_2, \ldots,n_m)$.
Similarly, points of the $m$-dimensional torus $\mathbb T^m$ will be denoted by $\vec t = (t_1, t_2, \ldots, t_m)$.
(We will identify $\TT$ with the unit interval mod $1$.)
For any $\vec t \in \TT^m$ we can define a character $\chi_{\vec t}$ of $H$ by
\[
	\chi_{\vec t}(\vec n) = e^{2\pi i \langle \vec t, \vec n\rangle}.
\]
Let $\operatorname{ind} \chi_{\vec t}$ be the $N$-dimensional representation of $G$ induced from the character $\chi_{\vec t}$.
That is, let $g_1=e, g_2, \ldots, g_N$ be all the elements of $G/H$.
Let $\xi_{g_i}$, $i=1,2, \ldots, N$, be an orthonormal basis in $\CC^N$.
Since $H$ is a normal subgroup, for each $\vec n \in H$ and each $g_i$ there is $\vec n'\in H$ such that
\[
	\vec n g_i = g_i \vec n'.
\]
By definition the induced representation then acts as
\[
	(\ind \chi_{\vec t})(\vec n) \xi_{g_i} = \chi_{\vec t}(\vec n')\xi_{g_i}.
\]
For each $i= 1, 2, \ldots, N$, the automorphism
\[
	\vec n \mapsto g_i^{-1}\vec ng_i, \quad \vec n \in H
\]
of $H$ is given by a matrix $A_i \in GL_m(\ZZ)$.
Moreover $A_i = A_j$ if and only if $g_i^{-1}\vec ng_i = g_j^{-1}\vec ng_j$ for all $\vec n\in H$ if and only if $g_jg_i^{-1}\in C(H)$, where $C(H)$ is the centralizer  of $H$.
Thus the restriction of $\operatorname{ind} \chi_{\vec t}$ onto $H$ is a direct sum of 1-dimensional representations
\begin{equation} \label{RestrictionOnH}
 (\ind \chi_{\vec t}) \vert_H = \diag(\overbrace{\chi_{\vec t} \circ A_1, \ldots, \chi_{\vec t} \circ A_1}^{k_1}, \ldots, \overbrace{\chi_{\vec t} \circ A_l, \ldots,  \chi_{\vec t} \circ A_l}^{k_l}).
\end{equation}

\noindent where $A_1, A_2, \ldots, A_l \in GL_m(\ZZ)$ are pairwise distinct, each $A_i$ is repeated $k_i$ times, $k_1+ \cdots + k_l = N$, and $A_1 = \myu$.
Since we assume that $G$ is non-abelian, we also have $l>1$.

We are now going to find a $*$-homomorphism $\alpha \colon C^*(G) \to C(\II^m, \matrM{N})$ as in the statement of the theorem.

Let $t_1^{(0)}, t_2^{0}, \ldots, t_m^{(0)} \in (0, 1)$ be such that $t_1^{(0)}, t_2^{(0)}, \ldots, t_m^{(0)}, 1$ are linearly independent over $\QQ$.
Let
\[
	\vec t^{(0)} = (t_1^{(0)}, t_2^{(0)}, \ldots, t_m^{(0)}) \in \TT^m.
\]
This will be the center of our ball.

For each pair $i \neq j$, we can find $\vec n^{(i, j)}\in H$ such that
\begin{equation}\label{1.1}
	(A_i-A_j) \vec n^{(i, j)} \neq 0,
\end{equation}
since distinct matrices must differ on vectors with integer coordinates.
Then for any integer $k$, $ \langle \vec t^{(0)}, (A_i-A_j)\vec n^{(i, j)}\rangle - k $ is a non-trivial linear combination of $t_1^{(0)}, t_2^{(0)}, \ldots, t_m^{(0)}, 1$ with integer coefficients.
Hence
\[
	\langle \vec t^{(0)}, (A_i-A_j)\vec n^{(i, j)}\rangle - k \neq 0
\]
and
\begin{equation}\label{1.2}
	\langle \vec t^{(0)}, (A_i-A_j)\vec n^{(i, j)}\rangle \neq 0 \mod 1.
\end{equation}
Since there are only finitely many $i$'s and $j$'s, it follows from (\ref{1.2}) that there exists $\delta >0$ such that for any $i\neq j$ and any integer $k$
\begin{equation}\label{1.9}
|\langle \vec t^{(0)}, (A_i-A_j)\vec n^{(i, j)}\rangle - k| > \delta.
\end{equation}
Choose a positive number $R$ such that
\begin{equation}\label{1.10}
	R \leq \min \left\{{\frac{\delta}{4 \max_{i\neq j} \|\vec n^{(i,j)}\| \max_{i} \|A_i\| }, \frac{1}{4\max_i \|A_i\| }} \right\},
\end{equation}
and such that the ball $B = B(\vec t^{(0)}, R)$ does not intersect the boundary $[0, 1]^m$.
Thus we can view $B$ as a subset of $\TT^m$.
Let $r \colon C(\TT^m) \to C(\II^m)$ be the restriction map onto $B$ composed with a homomorphism of $B$ and $\II^m$.
Define a $*$-homomorphism $\beta \colon C^*(G) \to C(\TT^m, \matrM{N})$ by
\[
	\beta(g)(\vec t) = (\operatorname{ind} \chi_{\vec t})(g).
\]
Define a $*$-homomorphism $\alpha \colon C^*(G) \to C(\II^m) \otimes \matrM{N}$ by
\[
	\alpha = (r\otimes_{\operatorname{id}_{\matrM{N}}})\circ \beta.
\]
Define $\mathcal A \subseteq C(\II^m) \otimes \matrM{N}$ by

\[
	\mathcal A = \setof{ \diag ( \overbrace{\psi_1, \ldots, \psi_1}^{k_1}, \ldots, \overbrace{\psi_l, \ldots, \psi_l )}^{k_l}}{
	\begin{array}{c} \psi_i \in C(\II^m), \\ \text{ each }  \psi_i  \text{ is repeated }  k_i  \text{ times} \end{array}
	}
\]

{\bf Claim:} $\alpha(C^*(H)) = \mathcal A.$

{\bf Proof of Claim:} It follows from (\ref{RestrictionOnH}) that $\alpha(C^*(H)) \subseteq \mathcal A.$ Since $\mathcal A$ is a commutative $C^*$-algebra isomorphic to $C(\II^m \sqcup \ldots \sqcup \II^m)$ (disjoint product of $l$ copies of $\II^m$), we can apply Stone-Weierstrass theorem to prove the claim. It means to prove two statements:

1) For any $i$ and any $\vec t \neq \vec t'\in \II^m$ there exists $\vec n \in H$ such that
\[
	\left(\alpha(\vec n)(\vec t)\right)_{ii} \neq \left(\alpha(\vec n)(\vec t')\right)_{ii}
\]
(this corresponds to separating points $\vec t \neq \vec t' $ inside the i-th copy of $\II^m$).

2) For any $\vec t, \vec t'\in \II^m$ and any $i\neq j$ there exists $\vec n\in H$ such that
\[
	\left(\alpha(\vec n)(\vec t)\right)_{ii} \neq \left(\alpha(\vec n)(\vec t')\right)_{jj}
\]
(this corresponds to separating points $\vec t,  \vec t' $ inside two different copies of $\II^m$).

Proof of 1):  Since $\vec t \neq \vec t'$ and $A_i$ is invertible, $A_i^*(\vec t- \vec t')\neq 0$ and hence there is $\vec n$ with one coordinate being equal to 1 and others being zero such that
\begin{equation}\label{1.4} \langle \vec t - \vec t', A_i \vec n\rangle = \langle A_i^*(\vec t - \vec t'),  \vec n\rangle \neq 0. \end{equation}
Also \begin{equation}\label{1.5}|\langle \vec t - \vec t', A_i \vec n\rangle| \le \|\vec t-\vec t'\| \|A_i\| < 1\end{equation} by (\ref{1.10}). It follows from (\ref{1.4}) and (\ref{1.5})
that
\[
	\langle \vec t - \vec t', A_i \vec n\rangle \neq 0 \mod 1.
\]
 Hence
\begin{align*}
 |\left(\alpha(\vec n)(\vec t)\right)_{ii} - \left(\alpha(\vec n)(\vec t')\right)_{ii}| &= |\chi_{\vec t}(A_i\vec n) - \chi_{\vec t'}(A_i\vec n)| \\
 &= |e^{2\pi i \langle \vec t, A_i \vec n\rangle} - e^{2\pi i \langle \vec t', A_i \vec n\rangle} | \\
 &= | e^{2\pi i \langle \vec t - \vec t', A_i\vec n\rangle} - 1| \neq 0.
\end{align*}

Proof of 2): Let $\vec n^{(i, j)}$ be as in (\ref{1.1}). Then for any integer $k$
\begin{align*}
|\langle \vec t, A_i &\vec n^{(i, j)}\rangle - \langle \vec t', A_j \vec n^{(i, j)}\rangle -k| \\
&=|\langle \vec t - \vec t', A_i \vec n^{(i, j)} \rangle - \langle \vec t'-\vec t, A_j \vec n^{(i, j)}\rangle + \langle \vec t^{(0)}, A_i \vec n^{(i, j)} - A_j \vec n^{(i, j)}\rangle - k|  \\
&\ge |\langle \vec t^{(0)}, A_i \vec n^{(i, j)} - A_j \vec n^{(i, j)}\rangle - k| - 2 r \max_i\|A_i\|\max_{i\neq j}|\vec n^{(i,j)}| \\
&\ge \delta - \delta/2 > 0
\end{align*}
by (\ref{1.9}) and (\ref{1.10}). Thus
\[
	\langle \vec t, A_i \vec n^{(i, j)}\rangle - \langle \vec t', A_j \vec n^{(i, j)}\rangle \neq 0 \mod 1
\]
and we obtain
\begin{align*}
	|\left(\alpha(\vec n)(\vec t)\right)_{ii} - \left(\alpha(\vec n)(\vec t')\right)_{jj}| &= |\chi_{\vec t}(A_i\vec n^{(i, j)}) - \chi_{\vec t'}(A_j\vec n^{(i, j)})| \\
	&= |e^{2\pi i \langle \vec t, A_i\vec n^{(i, j)}\rangle} - e^{2\pi i \langle \vec t', A_j\vec n^{(i, j)}\rangle}| \\
	& = |e^{2\pi i(\langle \vec t, A_i \vec n^{(i, j)}\rangle - \langle \vec t', A_j \vec n^{(i, j)}\rangle)} - 1| \neq 0.
\end{align*}
This proves the claim.

\noindent Since $C(\II^m, \mathbb C \myu_{\matrM{N})} \subseteq \mathcal A$, the statement of the theorem follows from the claim.
\end{proof}

\begin{lemma}\label{tensor} Let $\Omega_n$, $S_n$ be the Voiculescu matrices and let $F$ be a finite-dimensional $C^*$-algebra. Then there are no commuting matrices $A_n, B_n\in \matrM{n}\otimes F$ such that
\[
	\|A_n - \Omega_n \otimes \myu_F\|\to 0, \|B_n - S_n \otimes \myu_F\|\to 0
\]
as $n\to \infty$.
\end{lemma}
\begin{proof} Exel and Loring proved that there is $d>0$ such that any pair of commuting matrices in $\matrM{n}$ is at distance not less than $d$ from $(\Omega_n, S_n)$. Their proof works for $(\Omega_n\otimes \myu_F, S_n\otimes \myu_F)$ as well, even with the same constant $d$.
\end{proof}

\begin{lemma}\label{Fredholm}
Let $H$ be a separable Hilbert space with basis $\{e_i\}$. Let $T_n \in \mathbf{B}(H)$, $n\in \mathbb N$, be given by
\[
	T_n e_k = \begin{cases}
       \frac{k}{2n^2}e_{k+1}, &  1\le k\le 2n^2\\
       e_{k+1}, &  k>2n^2
     \end{cases}
\]
Let $F$ be a finite-dimensional $C^*$-algebra. Then there are no normal elements $N_n\in \mathbf{B}(H)\otimes F$ such that
\[
	\|N_n - T_n\otimes \myu_F\| \to 0,
\]
as $n\to \infty$.
\end{lemma}
\begin{proof}
Write $F$ as
\[
	F = \matrM{k_1} \oplus \cdots\oplus \mathbf{M}_{k_N}
\]
and let $k = k_1 + \cdots + k_N$. We can consider elements of $\mathbf{B}(H) \otimes F$ as operators on $\bigoplus_{i=1}^k H$.
Suppose there are normal elements $N_n\in \mathbf{B}(H)\otimes F$ such that \[
	\|N_n - T_n\otimes \myu_F\| \to 0
\]
as $n\to \infty$. Since each $T_n$, $n\in \mathbb N$, is a finite rank  perturbation of the unilateral shift $T$, this would imply that
\[
	\|N_n - T\otimes \myu_F\|_e \to 0
\]
as $n\to \infty$ (here $\| \|_e$ is the essential norm). Since the Fredholm index does not change under small perturbations, this would imply that for $n$ large enough, $N_n$ is Fredholm and $\ind(N_n) = \ind (T\otimes \myu_F) = - k.$ However, the Fredholm index of any normal Fredholm operator is zero.
\end{proof}

\begin{theorem}\label{VirtuallyAbelianMWSP}  Let $G$ be a finitely generated virtually abelian group of rank  $m$,  where $m\ge 3$. Then $G$ is not matricially stable.
\end{theorem}

\begin{proof} At first we notice that $G$ must contain $\mathbb Z^m$ as a normal subgroup of finite index. Indeed $G$ contains an abelian subgroup $\widetilde H \cong \mathbb Z^m$ of finite index.
Then $H = \bigcap_{g\in G}g\widetilde H g^{-1}$ is normal abelian subgroup of $G$ of finite index. Hence its rank  is $m$. On the other hand $H$ is a subgroup of $\widetilde H \cong \mathbb Z^m$ and hence is isomorphic to $\mathbb Z^k$, for some $k$. Hence $k=m$.

Let $N$ be the index of $H$.
By Theorem \ref{MapToScalarFunctions} there is a $*$-homomorphism
$\alpha\colon C^*(G) \to C(\II^m, \matrM{N})$ such that
\[
	\alpha(C^*(H))\supseteq C(\II^m, \mathbb C \myu_N).
\]
Since $m\ge 3$, $\II^m$ contains a homeomorphic copy of two-dimensional torus which we will denote by $\mathbb T^2$.  Let $\tilde \alpha\colon C^*(G) \to C(\mathbb T^2, \matrM{N})$ be the composition
of $\alpha$ with the restriction map $ C(\II^m, \matrM{N}) \to C(\mathbb T^2, \matrM{N}).$
Then
\begin{equation}\label{surjective}
	\tilde \alpha(C^*(H)) \supseteq C(\mathbb T^2, \mathbb C \myu_N).
\end{equation}

 Let $e^{2 \pi i x}$ and $e^{2 \pi i y}$ be the standard generators of $C(\mathbb T^2)$ and let $\tilde \pi \colon C(\mathbb T^2) \to \prod \matrM{n} / \bigoplus \matrM{n}$ be given by
\[
	\tilde\pi(e^{2 \pi i x}) = q\left((\Omega_n)_{n\in \mathbb N}\right), \tilde\pi(e^{2 \pi iy}) = q\left((S_n)_{n\in \mathbb N}\right),
\]
where $q \colon \prod \matrM{n} \to \prod \matrM{n} /\bigoplus \matrM{n}$ is the canonical surjection and $\Omega_n, S_n$ are the Voiculescu matrices.  Let
\[
	\pi = \tilde \pi \otimes \operatorname{id}_{\matrM{N}} \colon C(\mathbb T^2, \matrM{N}) \to \prod (\matrM{n}\otimes \matrM{N}) \left/ \bigoplus (\matrM{n}\otimes \matrM{N}) \right..
\]

Suppose $G$ is matricially stable. Then $\pi\circ \tilde \alpha$ lifts to some $*$-homomorphism $\psi = (\psi_n)_{n\in \mathbb N}$, where $\psi_n\colon C^*(G) \to \matrM{n} \otimes \matrM{N}$.
By (\ref{surjective}) there are $a, b\in C^*(H)$  such that
\[
	\alpha(a) = e^{2 \pi i x}\otimes \myu_N, \quad \text{ and } \quad  \alpha(b) = e^{2 \pi i y}\otimes \myu_N.
\]
 It implies that
\begin{align*}
	\|\psi_n(a) - \Omega_n\otimes \myu_N\| \to 0, \\
	\|\psi_n(b) - S_n\otimes \myu_N\| \to 0.
\end{align*}
Since $C^*(H)$ is a commutative $C^*$-subalgebra of $C^*(G)$, $\psi_n(a)$ and  $\psi_n(b)$ commute, for each $n\in \mathbb N$. This contradicts  Lemma \ref{tensor}.
\end{proof}

\begin{theorem}\label{VirtAbWSP} Let $G$ be a finitely generated virtually abelian group. The following are equivalent:
\begin{enumerate}[(i)]
\item $G$  is $C^*$-stable

\item $G$ is weakly $C^*$-stable

\item  $G$ has rank  less or equal to one.
\end{enumerate}
\end{theorem}
\begin{proof}  (i)$\Rightarrow$(ii) is obvious.

(ii)$\Rightarrow$(iii): Suppose the rank  of $G$ is bigger than one. In the same way as in Theorem \ref{VirtuallyAbelianMWSP}
we conclude that $G$ contains  a normal subgroup $H\cong\mathbb Z^m$ of finite index, where $m\ge 2$.
 Let $N$ be the index of $H$.
By Theorem \ref{MapToScalarFunctions} there is a $*$-homomorphism
$\alpha\colon C^*(G) \to C(\II^m, \matrM{N})$ such that
\[
	\alpha(C^*(H))\supseteq C(\II^m, \mathbb C \myu_N).
\]
Since $m\ge 2$, $\II^m$ contains a homeomorphic copy of the two-dimensional disc which we will denote by $\mathbb D$.  Let $\tilde \alpha\colon C^*(G) \to C(\mathbb D,  \matrM{N})$ be the composition
of $\alpha$ with the restriction map $ C(\II^m, \matrM{N}) \to C(\mathbb D, \matrM{N}).$
Then
\begin{equation}\label{surjective2}\tilde \alpha(C^*(H)) \supseteq C(\mathbb D, \mathbb C \myu_N).\end{equation}

Let $T_n$ be as in Lemma \ref{Fredholm}.    It is easy to check that \begin{equation}\label{almostnormal}\|\acomm{T_n}{T_n^*}\|\to 0\end{equation} and hence we can define a $*$-homomorphism  $\tilde \pi\colon C(\mathbb D) \to \prod \mathbf{B}(H) /\bigoplus \mathbf{B}(H)$ by
\[
	\tilde\pi(z) = q\left((T_n)_{n\in \mathbb N}\right),
\]
where $z \in C(\mathbb D)$ is the identity function and  $q\colon \prod \mathbf{B}(H) \to \prod \mathbf{B}(H)/\bigoplus \mathbf{B}(H)$ is the canonical surjection.
Let
\[
	\pi = \tilde \pi \otimes \operatorname{id}_{\matrM{N}}\colon C(\mathbb D, \matrM{N}) \to \prod (\mathbf{B}(H)\otimes \matrM{N}) /\bigoplus (\mathbf{B}(H)\otimes \matrM{N}).
\]
If $G$ was weakly $C^*$-stable then $\pi\circ \tilde\alpha$ would lift to some $*$-homomorphism $\psi = (\psi_n)_{n\in \mathbb N}$, where $\psi_n\colon A \to \mathbf{B}(H) \otimes \matrM{N}$, $n\in \mathbb N$. By (\ref{surjective2}) there is $a\in C^*(H)$  such that
\[
	\tilde\alpha(a) = z\otimes \myu_N.
\]
It implies that
\[
	\|\psi_n(a) - T_n\otimes \myu_N\| \to 0.
\]
However since  $C^*(H)$ is a commutative $C^*$-subalgebra of $C^*(G)$, $\psi_n(a)$ is normal, for each $n\in \mathbb N$. This contradicts Lemma \ref{Fredholm}.

(iii) $\Rightarrow$ (i): If  the rank  of $G$ is less or equal to one, then $G$ is virtually free and hence is $C^*$-stable by Theorem \ref{vfissp}.
\end{proof}

As a corollary we obtain a characterization of stability for finitely generated abelian groups.

\begin{corollary}  Let $G$ be a finitely generated abelian group. The following are equivalent:
\begin{enumerate}[(i)]
\item $G$  is $C^*$-stable

\item $G$ is weakly $C^*$-stable

\item $G$ is matricially stable

\item  $G$ has rank  less or equal to one.
\end{enumerate}
\end{corollary}
\begin{proof} All implications but (iii)$\Rightarrow$(iv) follow from Theorem \ref{VirtAbWSP}.

(iii)$\Rightarrow$(iv): by Theorem \ref{VirtuallyAbelianMWSP} $G$ has rank less or equal to 2. If it is equal to 2, then $G = \mathbb Z^2 \times H$, where $H$ is a finite group.
By sending the generators of $\mathbb Z^2$ to the Voiculescu matrices and all elements of $H$ to identity matrices, we obtain a homomorphism from $C^*(G)$ to
$\prod \matrM{n} / \oplus \matrM{n}$. If it was liftable, a lifting would send the generators of $\mathbb Z^2$ to commuting elements. This would contradict to Lemma \ref{tensor}.
\end{proof}

Stability of non-finitely generated countable abelian groups is much more involved and will be considered in a sequel paper.

\section{Classes of groups} \label{sec:classes}

\subsection{Crystallographic groups}\label{Crystal groups}

Among the virtually abelian finitely generated groups, the \emph{crystallographic groups} form a prominent class, since they describe all possible symmetry groups of lattices in $\RR^n$. By a theorem of Zassenhaus (\cite{hz:abr}), such groups may be characterized abstractly as those groups $G$ for which $\ZZ^n$ is a subgroup of finite index so that $\ZZ^n$ becomes maximal among all abelian subgroups of $G$. In this case $\ZZ^n$ is automatically normal, and the quotient group $D=G/\ZZ^n$ is called the \emph{point group}.

There are only finitely many crystallographic groups in each dimension, as proved by Bieberbach (\cite{lb:ber}) answering a part of Hilbert's 18th problem. At $n=1$ we have the two \emph{line groups} $\ZZ$ and $\ZZ_2 \star \ZZ_2$ which we already have seen are \SP, and at $n\geq 3$ we know by Theorem \ref{VirtuallyAbelianMWSP} that none of the stabilty properties are met.
  At $n=2$ we have the seventeen \emph{wallpaper groups} of which we have already encountered $\crystal{p1}=\ZZ^2$ and $\crystal{pg}$  and noted that $\crystal{p1}$ is not \MWSP\ but  $\crystal{pg}$ is, and that neither is \WSP. Thus, we see that stability properties are not determined by dimension alone. Invoking $K$-theory, we will show that exactly twelve of the wallpaper groups are \MWSP.

To do so, we will establish lack of matricial stability by Theorem \ref{exelloring} in five cases and establish matricial stability in the remaining twelve by appealing to an old result by Loring, Pedersen and the first named author (\cite[8.2.2(ii)]{elp:sar}) which shows that any \emph{2-dimensional NCCW complex} which has only torsion infinitesimals in its ordered $K_0$-group has the desired stability property. Recall that an element $[x]$ of the $K$-group $K_0(A)$ of a $C^*$-algebra $A$ is called an {\it infinitesimal}  if
\[
	-[1_A] \le n[x] \le [1_A]
\]
for any $n\in \mathbb Z$. Of course $0$ is an infinitesimal, and any $K_0$-map induced by a unital $*$-homomorphism will map infinitesimals to infinitesimals.

The class of $n$-dimensional {\it noncommutative CW  (NCCW) complexes}  is defined recursively by saying that the $0$-dimensional ones are exactly the finite-dimensional $C^*$-algebras, and that the $n$-dimensional ones are pullbacks of the form
\[
A_n=C(\II^n,F_n)\oplus_{C(\partial\II^n,F_n)}A_{n-1}
\]
where $A_{n-1}$ is an $(n-1)$-dimensional NCCW, $F_n$ is finite-dimensional, and the pullback is taken over the canonical map from $C(\II^n,F_n)$ to $C(\partial\II^n,F_n)$ on one hand, and an arbitrary  unital $*$-homomorphism $\gamma_n \colon A_{n-1}\to C(\partial\II^n,F_n)$ on the other.

In the 12 positive cases, we draw extensively on previous work \cite{yang, mcalister} which provide concrete realizations of  $C^*(G)$ as a 2-dimensional NCCW for any wallpaper group $G$. The thesis \cite{yang} also provides complete descriptions of the  $K$-groups, but  the order structures was not computed there. In some cases, like $\crystal{pg}$, it is easy to see that there are no infinitesimals in the $K_0$-group, but in general such computations are rather demanding. Fortunately, as pointed out in  \cite{mcalister}, the $K$-theory of NCCW complexes of dimension two or less can be systematically computed by appealing to the Mayer-Vietoris  exact sequence
\[
\xymatrix{
{K_1(A_n)}\ar[r]&{K_1(A_{n-1})}\ar[r]&{K_1(C(\partial\II^n,F_n))}\ar[d]^-{\delta_n}\\
{K_0(C(\partial\II^n,F_n))}\ar[u]&{K_0(A_{n-1})\oplus K_0(C(\II^n,F_n))}\ar[l]&{K_0(A_n)}\ar[l]^-{\phi_n}
}
\]
and we will use this to provide a more direct path to establishing that there are only torsion infinitesimals in these cases. We do not need to specify the maps in the Mayer-Vietoris sequence, except to note that $\phi_n$ is induced by a unital $*$-homomorphism (the canonical restriction map) and hence sends infinitesimals to infinitesimals.

\begin{theorem}\label{forelp}
Let $A_2=C(\II^2,F_2)\oplus_{C(\partial\II^2,F_2)}A_{1}$ be a $2$-dimensional NCCW complex with $s$ sheets, i.e. with $F_2=\bigoplus_{i=1}^s\mathbf{M}_{n_i}$ and consider the statements
\begin{enumerate}[(i)]
\item $\rank K_0(A_1)=\rank K_0(A_2)$
\item $\rank K_1(A_1)-\rank K_1(A_2)=s$
\item All infinitesimals of $K_0(A_2)$ are torsion
\item $A_2$ is matricially stable.
\end{enumerate}
Then
\[
(i)\Longleftrightarrow (ii)\Longrightarrow (iii)\Longrightarrow(iv)
\]
\end{theorem}
\begin{proof}
Note first that at $n=1$, the Mayer-Vietoris sequence degenerates to
\[
\xymatrix{
0\ar[r]&K_0(A_1)\ar[r]^-{\phi_1}&K_0(A_0)\oplus K_0(F_1)\ar[r]&K_0(F_1)^2\ar[r]&K_1(A_1)\ar[r]&0}
\]
with $\phi_1$ sending infinitesimals to infinitesimals. But since $K_0(A_0)\oplus K_0(F_1)$ has only the trivial infinitesimal, the same is the case for $K_0(A_1)$. At $n=2$, the Mayer-Vietoris sequence becomes
\[
\xymatrix{
{K_1(A_2)}\ar[r]&{K_1(A_{1})}\ar[r]&{K_0(F_2)}\ar[d]^-{\delta_2}\\
{K_0(F_2)}\ar[u]^-0&{K_0(A_{1})\oplus K_0(F_2)}\ar[l]&{K_0(A_2)}\ar[l]^-{\phi_2}
}
\]
where the upward map vanishes since it is easy to see that the preceding map is a surjection. Again we see that $K_0(A_{1})\oplus K_0(F_2)$ has no infinitesimals, so any infinitesimals in $K_0(A_2)$ must lie in $\operatorname{Im}{(\delta_2)}$.

Now we note that $ K_0(F_2)=\ZZ^s$ and rationalize to get
\[
\xymatrix{
{K_1(A_2)\otimes\QQ}\ar[r]&{K_1(A_1)\otimes\QQ}\ar[r]&{\QQ^s}\ar[d]^-{\delta_2\otimes\operatorname{id}_{\QQ}}\\
{\QQ^s}\ar[u]^-0&{(K_0(A_1)\otimes\QQ)\oplus \QQ^s}\ar[l]&{K_0(A_2)\otimes\QQ}\ar[l]
}
\]
and see that (i) and (ii) are equivalent by comparing dimensions in the upper and lower row. By the same argument, these conditions imply that $\delta_2\otimes\operatorname{id}_{\QQ}=0$, whence $\operatorname{Im}(\delta_2)\subset \operatorname{Tor}(K_0(A_2))$, and we have proved that (i) and (ii) imply (iii). The last implication is \cite[8.2.2(ii)]{elp:sar}.
\end{proof}

\begin{figure}
\begin{center}
%%%%%%%%%%%
%%
% Mulitiline equations in tabular as explained here:
% https://tex.stackexchange.com/questions/62939/type-multiline-equation-in-a-table-in-latex
%%
%%%%%%%%%%%%

\renewcommand{\arraystretch}{1.3}
\begin{tabular}{|c||l|c|c|c|c|c|}\hline
&Generators&$F$&\multicolumn{2}{c|}{$K_0(A_i)$}&\multicolumn{2}{c|}{$K_1(A_i)$}\\\hline\hline
\cellcolor{black!20}{$\crystal{p1}$}&$xy=yx$&$\CC^2$&$\ZZ^2$&$\ZZ$&$\ZZ^2$&$\ZZ^2$\\\hline

\cellcolor{black!20}{ $\crystal{p2}$}&$t^2_i=(t_1t_2t_3)^2=e$&$\matrM{2}^2$&$\ZZ^6$&$\ZZ^5$&$0$&$\ZZ$\\\hline

$\crystal{pm}$&$r_1^2=r_2^2=e, r_iy=yr_i$&$\matrM{2}$&$\ZZ^3$&$\ZZ^3$&$\ZZ^3$&$\ZZ^4$\\\hline

$\crystal{pg}$&$p^2=q^2$&$\matrM{2}$&$\ZZ$&$\ZZ$&$\!\ZZ+\ZZ_2\!$&$\!\ZZ^2+\ZZ_2\!$\\\hline

$\crystal{cm}$&$r^2=e,rp^2=p^2r$&$\matrM{2}$&$\ZZ^2$&$\ZZ^2$&$\ZZ^2$&$\ZZ^3$\\\hline

$\crystal{pmm}$&
$\!\!\!\begin{array} {r@{}l@{}}
	r^2 & {}= (r_1r_2)^2=(r_2r_3)^2\\
	    & {}= (r_3r_4)^2=(r_4r_1)^2=e
\end{array}$
&$\matrM{4}$&$\ZZ^9$&$\ZZ^9$&$0$&$\ZZ$\\\hline

$\crystal{pmg}$&
$\!\!\!\begin{array} {r@{}l@{}}
	& {} r^2=t_1^2=t_2^2=e,\\
	{} &t_1rt_1=t_2rt_2
\end{array}$
&$\matrM{4}$&$\ZZ^4$&$\ZZ^4$&$\ZZ$&$\ZZ^2$\\\hline

$\crystal{pgg}$&$(po)^2=(p^{-1}o)^2=e$&$\matrM{4}$&$\ZZ^3$&$\ZZ^3$&$\ZZ_2$&$\ZZ+\ZZ_2$\\\hline

$\crystal{cmm}$&
$\!\!\!\begin{array} {r@{}l@{}}
	r_i^2 & {}=t^2=(r_1r_2)^2 \\
	& {}=(r_1tr_2t)^2=e
\end{array}$
&$\matrM{4}$&$\ZZ^5$&$\ZZ^5$&$0$&$\ZZ$\\\hline

\cellcolor{black!20}{$\crystal{p4}$}&$r^4=t^2=(rt)^4=e$&$\matrM{4}^2$&$\ZZ^9$&$\ZZ^8$&$0$&$\ZZ$\\\hline

$\crystal{p4m}$&
$\!\!\!\begin{array} {r@{}l@{}}
	r_i^3 & {}=(r_1r_2)^4=(r_2r_3)^4\\
	& {}=(r_3r_1)^2=e
\end{array}$
&$\matrM{4}$&$\ZZ^9$&$\ZZ^9$&$0$&$\ZZ$\\\hline

$\crystal{p4g}$&$t^4=r^2=(t^{-1}rtr)^2=e$&$\matrM{8}$&$\ZZ^6$&$\ZZ^6$&$0$&$\ZZ$\\\hline

\cellcolor{black!20}{$\crystal{p3}$}&$r^3=t^3=(rt)^3=e$&$\matrM{3}^2$&$\ZZ^8$&$\ZZ^7$&$0$&$\ZZ$\\\hline

$\crystal{p3m1}$&
$\!\!\!\begin{array} {r@{}l@{}}
	r_i^2 & {}=(r_1r_2)^3=(r_2r_3)^3\\
	& {}=(r_3r_1)^3=e
\end{array}$
&$\matrM{6}$&$\ZZ^5$&$\ZZ^5$&$\ZZ$&$\ZZ^2$\\\hline

$\crystal{p31m}$&$t^3=r^2=(t^{-1}rtr)^3=e$&$\matrM{6}^2$&$\ZZ^5$&$\ZZ^5$&$\ZZ$&$\ZZ^3$\\\hline

\cellcolor{black!20}{$\crystal{p6}$}&$r^3=t^2=(rt)^6=e$&$\matrM{6}^2$&$\ZZ^{10}$&$\ZZ^9$&$0$&$\ZZ$\\\hline

$\crystal{p6m}$&
$\!\!\!\begin{array} {r@{}l@{}}
	r_i^2 & {}=(r_1r_2)^3=(r_2r_3)^6\\
	& {}=(r_3r_1)^2=e
\end{array}$
&$\matrM{12}$&$\ZZ^8$&$\ZZ^8$&$0$&$\ZZ$\\\hline

\end{tabular}
\end{center}
\caption{The 17 wallpaper groups. \textsl{The table is ordered as Table 3 in \cite{coxmos} and the generators given are taken from there (some generators' names are substituted to avoid notational clashes). The groups in shaded cells fail to be matricially stable. The $K$-groups given are for $i=2$ to the left and for $i=1$ to the right.}}\label{sweet17}
\end{figure}

\begin{corollary} \label{cor:mwspcrystal}
All of the groups  $\crystal{cm}$,  $\crystal{pm}$,  $\crystal{pg}$,  $\crystal{cmm}$, $\crystal{pmm}$,  $\crystal{pmg}$,  $\crystal{pgg}$,  $\crystal{p3m1}$,  $\crystal{p31m}$, $\crystal{p4mm}$, $\crystal{p4mg}$, $\crystal{p6mm}$ are \MWSP.
\end{corollary}
\begin{proof}
The $K$-groups of the presentation $C^*(G)$ as a $2$-dimensional NCCW $A_2$ given in \cite{yang} is given in Figure \ref{sweet17}. Theorem \ref{forelp} finishes the argument by checking condition (i) or (ii).
\end{proof}

\begin{theorem} \label{thm:NonMWSPcrystals}
Neither of the groups  $\crystal{p1}$,  $\crystal{p2}$,  $\crystal{p3}$,  $\crystal{p4}$, $\crystal{p6}$ are \MWSP.
\end{theorem}
\begin{proof}
 \crystal{p2} is defined by the relations $t_i^2 = (t_1t_2t_3)^2 = e.$ Since $t_i^2 = e$,  the relation $(t_1t_2t_3)^2 = e$ is equivalent to the homogeneous relation
\begin{equation} \label{0}
	t_1t_2t_3t_1^{-1}t_2^{-1}t_3^{-1}= e,
\end{equation}

and hence our Exel-Loring type invariant applies to this relation.

Let
\begin{align*}
 	T_{n,1} &= \begin{pmatrix} 0 & \Omega_n  \\ \Omega_n^{-1} & 0 \end{pmatrix}, \\
 	T_{n,2} &= \begin{pmatrix} 0 & S_n  \\ S_n^{-1} & 0 \end{pmatrix}, \\
 	T_{n,3} &= \begin{pmatrix} 0 & \myu_n  \\ \myu_n & 0 \end{pmatrix}.
 \end{align*}

Then for each $n \in \NN$,  $i = 1, 2,3 $, $T_{n, i}$ are unitaries and $T_i^2 = \myu_{2n}$.
It is easy to compute that
\[
	(T_{n, 1}T_{n, 2}T_{n, 3})^2 = \diag (\Omega_nS_n^{-1}\Omega_n^{-1}S_n, \Omega_n^{-1}S_n\Omega_nS_n^{-1}) = \overline\omega_n \myu_{2n}.
\]
Thus $\|(T_{n, 1}T_{n, 2}T_{n, 3})^2 - \myu_{2n}\| \to 0$ and hence $T_{n, 1}, T_{n, 2}, T_{n, 3}$ define an approximate representation of \crystal{p2}.
Since
\begin{align*} % hacked in alignment &
\wind \det (rT_{n, 1}&T_{n, 2}T_{n, 3}T_{n, 1}^{-1}T_{n, 2}^{-1}T_{n, 3}^{-1} + (1-r)\myu_{2n}) \\
 &= \wind \det (r(T_{n, 1}T_{n, 2}T_{n, 3})^2 + (1-r)\myu_{2n}) \\
 &= \wind (r \overline\omega_n + 1-r)^{2n} = -2 \neq 0,
\end{align*}
we conclude by Theorem \ref{exelloring} that $T_{n, 1}, T_{n, 2}, T_{n, 3}$ are not close to any unitaries satisfying (\ref{0}) and hence this approximate representation is not close to any actual representation of \crystal{p2}.

We now turn to  \crystal{p4}, using the presentation in Figure \ref{sweet17}.
Since $r = r^{-3}$ and $t = t^{-1}$, the relation $(rt)^4 = e$ is equivalent to the homogeneous relation
\begin{equation}\label{1}
	r^{-3}trtrt^{-1}rt^{-1}= e,
\end{equation}
and hence the methods above apply to this relation.
Let
\begin{align*}
	R_n &= \begin{pmatrix} 0 & 0 & 0& S_n \\ S_n & 0 & 0& 0 \\ 0 & S_n^{-1} & 0& 0 \\ 0 & 0 & S_n^{-1}& 0 \end{pmatrix}, \\
	T_n &= \begin{pmatrix} 0 & 0 & \Omega_n &0 \\ 0 & 0 & 0& \Omega_n  \\ \Omega_n^{-1} & 0 & 0& 0 \\ 0&\Omega_n^{-1} & 0 & 0 \end{pmatrix}.
\end{align*}
Then for each $n\in \mathbb N$, $R_n$ and $T_n$ are unitaries and $R_n^4 = T_n^2 = 1.$ It is easy to compute that
\begin{align*}
	(R_nT_n)^4 = & \diag (S_n\Omega_n^{-1}S_n\Omega_nS_n^{-1}\Omega_nS_n^{-1}\Omega_n^{-1},   S_n\Omega_nS_n^{-1}\Omega_nS_n^{-1}\Omega_n^{-1}S_n\Omega_n^{-1}, \\
		& \ S_n^{-1}\Omega_nS_n^{-1}\Omega_n^{-1}S_n\Omega_n^{-1}S_n\Omega_n,   S_n^{-1}\Omega_n^{-1}S_n\Omega_n^{-1}S_n\Omega_nS_n^{-1}\Omega_n)\\
		= & \ \overline \omega_n^2 \myu_{4n}.
\end{align*}
		
Thus $\|(R_nT_n)^4 -\myu_{4n}\|\to 0$ and hence $R_n, T_n$ define an approximate representation of  \crystal{p4}.
Since
\begin{multline*}
	\wind \det (rR^{-3}T_nR_nT_nR_nT_n^{-1}R_nT_n^{-1} + (1-r)\myu_{4n}) \\ = \wind \det (r(R_nT_n)^4 + (1-r)\myu_{4n})= \wind (r\overline \omega_n^2 + 1-r)^{4n} = -8 \neq 0,
\end{multline*}
we conclude by Theorem \ref{exelloring} that $T_n, R_n$ are not close to any unitaries satisfying (\ref{1}) and hence this approximate representation is not close to any actual representation of  \crystal{p4}.

For \crystal{p3} we similarly rewrite the last of the relations given in Figure \ref{sweet17} as the homogeneous relation
\[
	r^{-2}t^{-2}rtrt = e,
\]
and consider the matrices
\begin{align*}
	R_n &= \begin{pmatrix} 0 & 0 &  S_n^{-1} \\ S_n & 0& 0 \\ 0 & \myu_n & 0 \end{pmatrix}, \\
	T_n &= \begin{pmatrix} 0 & 0 & \myu_n \\ \Omega_n & 0 & 0  \\ 0 & \Omega_n^{-1} & 0 \end{pmatrix}.
\end{align*}
Then for each $n\in \mathbb N$, $R_n$ and $T_n$ are unitaries and $R_n^3 = T_n^3 = 1.$ It is easy to compute that
\begin{multline*}
	(R_nT_n)^3 = \diag (S_n^{-1}\Omega_n^{-1}S_n\Omega_n,   S_n\Omega_nS_n^{-1}\Omega_n^{-1},
  \Omega_nS_n^{-1}\Omega_n^{-1}S_n) =  \overline \omega_n \myu_{3n}.
 \end{multline*}
Thus $\|(R_nT_n)^3 -\myu_{3n}\|\to 0$ and hence $R_n, T_n$ define an approximate representation of  \crystal{p3}.
Since
\[
	\wind \det (r(R_nT_n)^3 + (1-r)\myu_{3n})= \wind (r\overline \omega_n + 1-r)^{3n} = -3 \neq 0,
\]
it follows from Theorem \ref{exelloring} that this approximate representation is not close to any actual representation of \crystal{p3}.

Finally, for  \crystal{p6} we rewrite to
\[
	r^{-2}tr^{-2}t^{-1}rtrt^{-1}rtrt^{-1}= e,
\]
and let
\begin{gather*}
	R_n = \begin{pmatrix} 0 & 0 & 0& 0& S_n^{-1} & 0 \\ 0 & 0 & 0& 0& 0 & S_n^{-1} \\ S_n & 0 & 0& 0& 0 & 0 \\ 0 & S_n & 0& 0& 0 & 0 \\ 0 & 0 & \myu_n & 0& 0 & 0 \\ 0 & 0 & 0& \myu_n& 0 & 0 \end{pmatrix}, \\
	T_n = \begin{pmatrix} 0 & 0 & 0& \Omega_n& 0 & 0 \\ 0 & 0 & 0& 0& \Omega_n & 0 \\ 0 & 0 & 0& 0& 0 & \Omega_n \\ \Omega_n^{-1} & 0 & 0& 0& 0 & 0 \\ 0 & \Omega_n^{-1} & 0 & 0& 0 & 0 \\ 0 & 0 & \Omega_n^{-1}& 0& 0 & 0 \end{pmatrix}.
\end{gather*}
Then for each $n\in \mathbb N$, $R_n$ and $T_n$ are unitaries and $R_n^3 = T_n^2 = 1.$ It is easy to compute that $(R_nT_n)^6 = \overline \omega_n^2 \myu_{6n}.$ Thus $\|(R_nT_n)^6 -\myu_{6n}\|\to 0$ and
hence $R_n, T_n$ define an approximate
representation of  \crystal{p6}. Since
\[
	\wind \det (r(R_nT_n)^6 + (1-r)\myu_{6n})= \wind (r\overline \omega_n^2 + 1-r)^{6n} = -12\neq 0,
\]
we conclude by Theorem \ref{exelloring} that this approximate representation is not close to any actual representation of \crystal{p6}.
\end{proof}

\begin{remark}
Just as what is the case for $\crystal{p1}=\ZZ^2$, it is possible to quantify the obstruction to matricial stability for representations of all $\crystal{pn}$ with $\mathfrak n=1,2,3,4,6$ by means of $K$-theory.
Indeed, one may choose a ``Bott-type'' element $x_0 \in K_0(C^*(\crystal{pn}))$ with the property that if $\phi \colon C^*(\crystal{pn}) \to \prod \matrM{n_i}/ \bigoplus \matrM{n_i}$ satisfies $\phi_*(x_0)=0$, then $\phi$ has a lift.
As in the classical case (cf.\ \cite{el:acum, elp:sar, setal:ccsr}) the obstruction can also be formulated for almost representations of the groups, using $K$-theory or winding numbers as desired.

In fact, the defining difference between the 5 groups lacking matricial stability and the 12 groups enjoying it is that the former are given by rotations alone, whereas the latter have at least one (glide) reflection.
In the 5 former cases the symmetries are hence all orientation-preserving, and the rational homology of the crystallographic group is the homology of some surface which is orientable, and hence has a nontrivial class in dimension 2.
This in turn leads to such a Bott-type element in $K_0$ of the group $C^*$-algebra by naturality of the Baum-Connes map.
In the 12 latter cases, the surface is not orientable, and the second homology vanishes.
Although our technical tools are not sufficiently precise to allow a formal proof neither of Theorem \ref{thm:NonMWSPcrystals} nor Corollary \ref{cor:mwspcrystal} along these lines, surely this is the underlying geometric explanation of the dichotomy.\footnote{\SESTART See the forthcoming paper \cite{DadarlatComing} for significant steps in that direction.\SEEND}
We are grateful to Andreas Thom for clarifying these matters for us, and for pointing out that similar phenomena occur for sofic groups in \cite{1801.08381}.
\end{remark}

\begin{remark}
Since we have proved all the forward implications in Theorem \ref{forelp} and since (iv) fails for the 5 remaining wallpaper groups, we have proved that all the conditions in Theorem \ref{forelp} are equivalent for the NCCW complexes associated to wallpaper groups. We suspect that may be true in general, but will not pursue that here.\footnote{\SESTART In  \cite{DadarlatComing}, Dadarlat establishes $(iv)\Longrightarrow(iii)$ of Theorem \ref{forelp} in wide generality.\SEEND} Note also that since the $K$-groups in the 12 positive cases are torsion free, in fact there are no infinitesimals at all. This implies that $C^*(G)$ are weakly stable with respect to the bigger class of $C^*$-algebras of stable rank one because of \cite[8.2.2(i)]{elp:sar}. Again, we have no use for that observation here.
\end{remark}

\subsection{Finitely generated torsion-free 2-step nilpotent groups}\label{Finitely generated torsion-free 2-step nilpotent groups}

We are going to use the following fact.

\begin{theorem}[{\cite[Theorem 4.1]{nilp}}] \label{thm:nilppresentaion}
Let $G$ be a group, and let $n, m$ be two integers greater than or equal to zero.
Then the following two statements are equivalent:

\begin{enumerate}
	\item $G$ admits a presentation of the form
	\[
	G = \genrel{ a_1, \ldots , a_n, c_1, \ldots ,c_m }{ \!\!\! \begin{array}{l} \mcomm{a_i}{a_j} = \prod_{t=1}^m c_t^{\lambda_t^{ij}}, \text{ for }  1 \leq i < j \leq n,  \\ \mcomm{a_i}{c_j} = \mcomm{ c_k}{c_r} = e, \text{ for all } i,j,k,r \end{array} \!\!\! }.
	\]
	\item $G$ is a finitely generated, torsion-free, 2-step nilpotent group satisfying
	\begin{gather*}
		\rank  (G/Z(G)) + \rank  Z(G) = n+m, \\
		\rank  G' \le m \le \rank  Z(G).
	\end{gather*}
\end{enumerate}
\end{theorem}

\begin{theorem}
Let $G$ be a finitely generated, torsion-free, 2-step nilpotent group.\footnote{\SESTART The equality of (b)--(d) is generalized beyond the 2-step case in \cite{DadarlatComing} by  combining \cite{DadarlatNew}  with classical results on the rational cohomology of nilpotent groups.\SEEND}
The following are equivalent.
\begin{enumerate}[(a)]
	\item $\rank  (G/Z(G)) + \rank  Z(G) \le 1$.
	\item $G \cong \ZZ$ or $G$ is the trivial group.
	\item $G$ is \SP .
	\item $G$ is \MWSP .
\end{enumerate}
\end{theorem}
\begin{proof}
It follows from \cite[Theorem 1.2]{saj:gring} that $(a) \implies (b)$, and the implications $(b) \implies (c) \implies (d)$ are clear.
We will establish $(d) \implies (a)$ by proving the contrapositive.
To this end let $G$ be a finitely generated, torsion-free, 2-step nilpotent group and suppose that $\rank  (G/Z(G)) + \rank  Z(G) > 1$.
Pick $n,m$ such that $G$ has a presentation as in Theorem \ref{thm:nilppresentaion}.
To show that $G$ is not \MWSP\ we consider the cases $n=0$, $n=1$, and $n \geq 2$ individually.

If $n=0$ then $m \geq 2$.
By sending $c_1, c_2$ to the Voiculescu matrices and the other $c_j$'s to the identity matrix, we get an approximate representation which is not close to a representation.

If instead $n=1$ we must have $m \geq 1$.
Then we send $a_1, c_1$ to the Voiculescu matrices and the other $c_j$'s to the identity matrix to get an approximate representation which is not close to a representation.

So assume $n \ge 2$.
If all $\lambda_t^{ij}=0$, then we can send $a_1, a_2$ to the Voiculescu matrices and all other generators to the identity and again get an approximate representation which is not close to a representation.
So we can assume that there exist $t_0, i_0, j_0$ such that $\lambda_{t_0}^{i_0 j_0}\neq 0$.
For ease of notation put $M = \lambda_{t_0}^{i_0 j_0}$.

Let $l \in \NN$ be odd and let $\omega_l = e^{2\pi i/l}$.
We will build an approximate representation of $G$ that is not close to an actual representation.
Define
\[
	A_l = \diag(\lambda_1, \lambda_2, \ldots, \lambda_l),
\]
where $\lambda_k = (\overline \omega_l)^{M\frac{k(k+1)}{2}}$, $k = 1, 2, \ldots, l$, and define
\[
	B_l = \diag(\mu_1, \mu_2, \ldots, \mu_l),
\]
where $\mu_k = (\overline {\omega_l})^{k}$, $k = 1, 2, \ldots, l$.

We observe that
\[
	\mcomm{A_l}{S_l} = \diag(\lambda_1 \overline{\lambda_l}, \lambda_2 \overline{\lambda_1}, \ldots, \lambda_l \overline{\lambda_{l-1}}).
\]
Since $l$ is odd we see that
\[
	\lambda_1\overline{\lambda_l} = \overline{\omega_l}^{M(1-\frac{l(l+1)}{2})} = \overline{ \omega_l}^M,
\]
as we also have $\lambda_k \overline{\lambda_{k-1}} = \overline{\omega_l}^{Mk}$, for $2 \leq k \leq l$, we conclude that
\begin{equation} \label{nilp1}
	\mcomm{A_l}{S_l} = \diag(\overline{\omega_l}^M, \overline{\omega_l}^{2M}, \ldots, \overline{\omega_l}^{lM}) = B_l^M.
\end{equation}
Since $\mu_1 \overline{\mu_l} = \mu_k \mu_{k-1} = \overline{\omega_l},$ when $k>1$, we obtain also
\begin{align}
\begin{split}\label{nilp3}
	\|\mcomm{B_l}{S_l} - \myu_{l}\| &= \left\|\diag( \mu_{1} \overline{\mu_l}, \mu_2 \overline{\mu_1}, \ldots, \mu_{l}\overline{\mu_{l-1}})- \myu_{l} \right\| \\
	 &= \left\|\diag(\overline{\omega_l} - 1, \overline{\omega_l} - 1, \ldots, \overline{\omega_l} - 1)\right\| \to 0.
\end{split}
\end{align}
Finally, since $A_n$ and $B_n$ are diagonal, we also have
\begin{equation}\label{nilp2}
	\mcomm{ A_n}{ B_n} = 1.
\end{equation}

By (\ref{nilp1}), (\ref{nilp3}), (\ref{nilp2}) we can define an approximate representation $\pi_l$ of $G$ by
\begin{align*}
	\pi_l(a_i) &= \begin{cases}
					A_l, & i = i_0 \\
					S_l, & i = j_0 \\
					\myu_l, & i \neq i_0, j_0
				\end{cases}, \\
	\pi_l(c_t) &= \begin{cases}
					B_l, & t = t_0 \\
					\myu_l, & t \neq t_0
				\end{cases}.
\end{align*}
Since
\[
 \wind  \det (r\mcomm{ B_l}{ S_l} + (1-r)\myu_l) = \wind (1+r(\overline \omega_l -1)))^l =  -1,
\]
it follows from Theorem \ref{exelloring} that $\pi_l(a_{j_0})$ and $\pi_l(c_{t_0})$ are not close to any commuting matrices and hence this approximate representation is not close to a representation.
\end{proof}

\subsection{Surface groups}\label{Surface groups}

\begin{theorem} For each $g\in \mathbb N$ the surface group
\[
	G_g = \genrel{ a_1, a_2, \ldots a_g, b_1, b_2, \ldots, b_g }{ \mcomm{ a_1}{ b_1} \cdots\mcomm{ a_g}{b_g}=e }
\]
is not \MWSP.
\end{theorem}
\begin{proof}
Let $X_{1,n} = \Omega_n$, $Y_{1,n} = S_n$, $X_{i,n} = Y_{i,n} = \myu_n$, $i\neq 1$. Then
\[
	\|\mcomm{ X_{1,n}}{Y_{1,n}} \cdots\mcomm{ X_{g,n}}{ Y_{g,n}} - \myu_n\| \to 0.
\]
Since
\[
	\wind \det (r \mcomm{ X_{1,n}}{ Y_{1,n}} \cdots\mcomm{ X_{g,n}}{ Y_{g,n}}  + (1-r)\myu_n)  = \wind (r e^{\frac{2 \pi i}{n}} + 1 - r)^n = 1 ,
\]
we get by Theorem \ref{exelloring} that these matrices are not close to matrices exactly satisfying the group relation.\footnote{\SESTART An alternative proof of this fact is provided in \cite{DadarlatComing} based on \cite{DadarlatNew}. This approach  uses a result of Lubotzky and Shalom that surface groups have RFD $C^*$-algebra and the fact that surface groups have the Haagerup property.\SEEND}
\end{proof}

%\begin{remark} We we informed by M. Dadarlat that this result can alternatively be obtained using his results from \cite{DadarlatNew}, \end{remark}

\subsection{Baumslag-Solitar groups}\label{Baumslag-Solitar groups}

We now turn our attention to the Baumslag-Solitar group.
For each pair of integers $n,m \in \ZZ$ the Baumslag-Solitar group $\baso(m,n)$ is defined as
\[
	\genrel{ a, b }{ ab^ma^{-1} = b^n }.
\]

\begin{theorem}
$\baso(1,-1)$ is matricially stable but not \WSP.
\end{theorem}
\begin{proof}
Since $\baso(1,-1)$ is isomorphic to  the crystallographic group
 \[
 \crystal{pg}= \genrel{ x,y }{ xy=y^{-1}x } = \genrel{ p,q }{ p^2=q^2 }
 \]
(via $p=x,q=xy$), which is virtually abelian of rank 2, by Theorem \ref{VirtAbWSP} $\baso(1,-1)$
is not weakly $C^*$-stable.
By Corollary \ref{cor:mwspcrystal} it is matricially stable.
\end{proof}

\begin{theorem}\label{Baumslag-SolitarB(m, m)}
For any integer $m$ the Baumslag-Solitar group $\baso(m, m)$ is not \MWSP.
\end{theorem}
\begin{proof}  Let $\omega_n = e^{2\pi i/n}$ and $\Omega_n, S_n\in \matrM{n}$ be the Voiculescu matrices.
Then
\[
	\Omega_n^mS_n\Omega_n^{-m}S_n^{-1} = \diag( \omega_n^m, \omega_n^m, \ldots, \omega_n^m )
\]
and
\[
	\|\Omega_n^mS_n\Omega_n^{-m}S_n^{-1} - \myu_n\| = |e^{2\pi im/n}-1| \to 0.
\]
We have
\begin{multline*}
	\wind \det (r\Omega_n^mS_n\Omega_n^{-m}S_n^{-1} + (1-r)\myu) = \wind (1+r(\omega_n^m-1)))^n =  m.\end{multline*}

Thus by Theorem \ref{exelloring} the distance from $\Omega_n, S_n$ to any pair of unitaries satisfying the group relation does not tend to zero.
\end{proof}

If $|n|,|m| > 1$ and $|n| \neq |m|$ then $\baso(n,m)$ is not RF by \cite[Theorem B]{sm:nrforg}.
Hence, if $\baso(n,m)$ is MF it would follow from Proposition \ref{prop:MFisNice} that it not \MWSP.
Despite not knowing if $\baso(n,m)$ is MF, we can still prove that in the case $n=2, m=3$ the group is not \MWSP.

\begin{theorem}\label{BS(2,3)}  $\baso(2,3)$ is not matricially stable.
\end{theorem}
\begin{proof} We start with an almost representation of $\baso(2,3)$ constructed in \cite{Radulescu}. We describe it here. Let $n\in \mathbb N$ and
$e_0, \ldots, e_{6n-1}$ be a basis in $\mathbb C^{6n}$.
Define $u_n, v_n\in \matrM{6n}$ by
\[
	u_ne_k = e^{2\pi i k/6n} e_k,
\]
for $k= 1, \ldots, 6n-1$ and
\begin{align*}
 v_ne_{2k} &= e_{3k} \\
 v_ne_{2k+1} &= e_{3k+1} \\
 v_ne_{2n+2k} &= e_{3k+2} \\
 v_ne_{2n+2k+1} &= e_{3n+3k+2} \\
 v_ne_{4n+2k} &= e_{3n+3k} \\
 v_ne_{4n+2k+1} &= e_{3n+3k+1},
\end{align*}
for $k = 0, \ldots, n-1.$
Since
\begin{align*}
	\{0, \ldots, 6n-1\} &= \bigsqcup_{k=0}^{n-1} \{2k, 2k+1, 2n+2k,2n+2k+1,4n+2k, 4n+2k+1\} \\
		&= \bigsqcup_{k=0}^{n-1} \{3k, 3k+1, 3k+2, 3n+3k, 3n+3k+1, 3n+3k+2\},
\end{align*}
$v_n$ is well-defined.
It was shown in \cite{Radulescu} that
\[
 	a\mapsto v_n, b\mapsto u_n
\]
is an almost representation of $\baso(2,3)$. We don't know if this almost representation is close to a representation. But we will construct another almost representation out of this one, which, as we will see, is not close to a representation.

Define $\tilde v_n\in \matrM{6n}$ by
\begin{align*}
	\tilde v_ne_0 &= \frac{1}{\sqrt 2}(e_0-e_{3n}),  \\
	\tilde v_ne_{3n} &= \frac{1}{\sqrt 2}(e_0+e_{3n}), \\
	\tilde v_n e_k &= e_k, \quad k \neq 0,3n.
\end{align*}
  Since $u_n^2$ is identity on $\Span\{e_0, e_{3n}\}$,
  \[
  	\acomm{\tilde v_n}{ u_n^2}=0.
  \]
  Clearly it follows that
  \[
  	\pi_n\colon a\mapsto \tilde v_n v_n, \;\; b\mapsto u_n
  \]
  is an almost representation. We claim that it is not close to any actual representation.

  In \cite{IrrRepr} all irreducible finite-dimensional representations of $\baso(2,3)$ are described (as well as those of some other Baumslag-Solitar groups). In particular it was shown that for any irreducible finite-dimensional  representation, there is a basis in which  the image of $a$ is a shift matrix and the image of $b$ is a diagonal matrix. It follows that all finite-dimensional irreducible representations send the multiplicative commutator $\mcomm{aba^{-1}}{b}$ to the identity.  Since any finite-dimensional representation is direct sum of irreducible ones, it follows that any finite-dimensional representation sends
   $\mcomm{aba^{-1}}{b}$ to the identity operator. Thus it is sufficient to prove that $\|\mcomm{\pi_n(aba^{-1})}{ \pi_n(b)} - 1_{6n}\|$   is far from zero. We calculate
 \[
	v_nu_nv_n^{-1}e_0 = e_0, \;\; v_nu_nv_n^{-1}e_{3n} = \lambda e_{3n},
\]
where $\lambda = e^{4\pi i/3}.$
Hence the restriction of $\acomm{\pi_n(aba^{-1})}{ \pi_n(b)}$ on $\Span\{e_0, e_{3n}\}$ is
  \begin{align*}
	\acomm{\pi_n(aba^{-1})}{ \pi_n(b)}&|_{\Span\{e_0, e_{3n}\}} \\
	&= \acomm{\tilde v_nv_nu_nv_n^{-1}\tilde v_n^{-1}}{ u_n}|_{\Span\{e_0, e_{3n}\}}  \\
	&=
  1/2\acomm{\begin{pmatrix}1& 1\\-1 & 1\end{pmatrix} \begin{pmatrix}1& 0\\0 & \lambda\end{pmatrix}\begin{pmatrix} 1& -1\\1 & 1\end{pmatrix}}{ \begin{pmatrix}1& 0\\0 & -1\end{pmatrix}} \\
  &= \begin{pmatrix}0& 1-\lambda\\\lambda-1 & 0\end{pmatrix}.
  \end{align*}
   Thus
   \[
   \|\mcomm{\pi_n(aba^{-1})}{ \pi_n(b)} - 1_{6n}\| = \|\acomm{\pi_n(aba^{-1})}{ \pi_n(b)}]\| \geq  |\lambda - 1| = |e^{4\pi i/3}-1|,
   \]
   for all $n$.
\end{proof}

Below we will strengthen the last result. In general matricial stability does not pass to subgroups of finite index as one can see by looking at virtually abelian groups.
However it is not possible for a \MWSP\ group to contain $\baso(2,3)$ as a finite index subgroup.
To prove this we introduce induced approximate representations.

Let $H$ be a finite index subgroup of $G$ and let $\pi_n\colon H \to \mathcal U(\matrM{k(n)})$ be an approximate representation of $H$. Let $N = [G:H]$. Let us denote for short $V_n = \mathbb C^{k(n)}$ and for each $i = 1, \ldots, N$ let $g_iV_n$ be an isomorphic copy of $V_n$. For any vector $x\in V_n$ the same vector in $g_iV_n$ we will write as $g_ix$. In particular it implies that $\|\sum_{i=1}^N g_ix_i\| = \sqrt{\sum_{i=1}^N \|x_i\|^2}.$ Let $g_1, \ldots, g_N$ be a full set of representatives in $G$ of the left cosets in $G/H$. Then for each $g\in G$ and each $g_i$ there is an $h_i$ in $H$ and $j(i)$ in $1,\ldots, N$ such that $gg_i = g_{j(i)}h$. We define the {\it induced approximate representation} $\Ind \pi_n\colon G \to \mathcal U(\matrM{Nk(n)})$ analogously to the notion of an induced representation:
\[
	(\operatorname{Ind} \pi_n)(g) \sum_{i=1}^N g_ix_i = \sum_{i=1}^N  g_{j(i)}\pi_n(h_i)x_i,
\]
where $x_i\in V$ for each $i$.

\begin{proposition}
Let $G$ be a group with a finite index subgroup $H$.
If $\pi_n$ is an approximate representation of $H$, then the induced approximate representation $\Ind \pi_n$, constructed above, is in fact an approximate representation of $G$.
\end{proposition}
\begin{proof} Let $g, g'\in G$. Then for any $g_i$ there is $h = h(i)\in H$ and $l=l(i)$ such that
\[
	g'g_i = g_lh.
\]
Now we find $h'=h'(i)\in H$ and $m=m(i)$ such that
\[
	gg_l = g_mh'.
\]
Then
\[
	gg'g_i = gg_lh = g_mh'h.
\]
It is straightforward to check that $m(i)\neq m(j)$ when $i\neq j$. For any vector $g_ix\in g_iV_n$ we have
\begin{multline*}
\left((\Ind \pi_n)(g)(\Ind \pi_n)(g')- (\Ind \pi_n)(gg')\right)g_ix \\
 = (\Ind \pi_n)(g)(g_l\pi_n(h)x) - g_m\pi_n(hh')x = g_m \left(\pi_n(h')\pi_n(h) - \pi_n(h'h)\right)x.
 \end{multline*}
Now for any vector $\sum_{i=1}^Ng_ix_i\in \bigoplus_{i=1}^N g_iV_n$ we obtain
\begin{align*}
% hacked in alignment & in the first line. Not the prettiest solution.
 \| (\Ind \pi_n)&(gg') - (\Ind \pi_n)(g)(\Ind \pi_n)(g') \| \\
 &= \sup_{\|\sum_{i=1}^N g_ix_i\|\le 1} \left\|\sum_{i=1}^N \left((\Ind \pi_n)(gg') - (\Ind \pi_n)(g)(\Ind \pi_n)(g')\right)g_ix_i \right\| \\
 &= \sup_{\|\sum_{i=1}^N g_ix_i\|\le 1} \left\|\sum_{i=1}^N g_{m(i)}\left(\pi_n(h')\pi_n(h) - \pi_n(h'h)\right)x_i\right\| \\
 &= \sup_{\|\sum_{i=1}^N g_ix_i\|\le 1} \sqrt{\sum_{i=1}^N \left\|g_{m(i)}\left(\pi_n(h')\pi_n(h) - \pi_n(h'h)\right)x_i\right\|^2} \\
 &\le  \left\|\left(\pi_n(h')\pi_n(h) - \pi_n(h'h)\right)\right\|\to 0.
\end{align*}
as $n\to \infty$.
\end{proof}

We saw above that some groups having $\ZZ^2=\baso(1,1)$ as a finite index subgroup could be \MWSP\ even though $\ZZ^2$ is not. For $\baso(2,3)$ this is not possible:

\begin{theorem} \label{thm:bssubgroup}
If a group $G$ contains $\baso(2,3)$ as a subgroup of finite index then $G$ is not matricially stable.
\end{theorem}
\begin{proof} In the proof of Theorem \ref{BS(2,3)} we constructed an approximate representation $\pi_n$ of $H=\baso(2,3)$ such that on the element $h_0 = \mcomm{aba^{-1}}{b}\in H$
\[
	\|\pi_n(h_0) - 1_{6n}\| \geq  |e^{4\pi i/3}-1|,
\]
for all $n$.
It follows from the definition of an induced approximate representation that the subspace $V_n$ of the space $\sum_{i=1}^N g_iV_n$ of the induced representation $\Ind \pi_n$ is invariant under $(\Ind \pi_n)(h)$, for all $h\in H$, and
\[
	(\Ind \pi_n)(h)\;|_{V_n} = \pi_n(h),
\]
for all $h\in H$. It implies that
\[
	\|(\Ind \pi_n)(h_0)- 1_{6nN}\| \ge \|\pi_n(h_0)- 1_{6n}\| \geq  |e^{4\pi i/3}-1|,
\]
for all $n$. On the other hand, as explained in the proof of Theorem \ref{BS(2,3)}, any representation of $H$ is identity on $h_0$. In particular for any representation $\rho$ of $G$, $\rho\;|_H(h_0)$ is identity. Hence $\Ind \pi_n$ is not close to any  representation of $G$.
\end{proof}

\begin{remark}
By examining the proof of Theorem \ref{thm:bssubgroup} we see that what it needs is an element $h_0 \in \baso(2,3)$ such that
\begin{enumerate}[(a)]
	\item \label{enum:bsh} $h_0$ is mapped to $\myu$ in any finite dimensional representation, and
	\item there is an approximate representation $\pi_n$ such that $\pi_n(h)$ is always far from $\myu$.
\end{enumerate}
We induce $\pi_n$ to an approximate representation $\Ind \pi_n$ of $G$.
By (b) we must have that $(\Ind \pi_n)(h_0)$ is far from $\myu$, so it follows from  (a) that $\Ind \pi_n$ cannot be close to an actual representation of $G$.
Hence the same proof would work for any group that satisfies (a) and (b).

We note that (\ref{enum:bsh}) is simply a restatement of the fact that $\baso(2,3)$ is not RF.
If we restrict to the class of MF groups we can use this fact directly to extend Theorem \ref{thm:bssubgroup} drastically.
Indeed, suppose $H$ is a subgroup of $G$ and that they both are MF.
If $H$ is not RF, then $G$ is not RF, so if $G$ is finitely generated then it cannot be \MWSP\ by Proposition \ref{prop:MFisNice}.
Note that $G$ is necessarily finitely generated if $H$ is finitely generated and the index of $H$ in $G$ is finite.
\end{remark}

\section{Matricial stability versus Hilbert-Schmidt stability}
The normalized Hilbert-Schmidt norm seems to be more friendly for stability questions than the operator norm.
For example all virtually abelian groups are Hilbert-Schmidt stable as was proved in \cite{HadwinShulmanGroups} while many of them are not matricially stable as we saw above.

In general we do not know if matricial stability implies Hilbert-Schmidt stability.
Below we show that this is the case for amenable groups.

Recall from \cite{TWW} that a trace $\tau$ on a $C^*$-algebra $A$ is {\it quasidiagonal} if  for every finite set $F$ of A and $\epsilon > 0$,
there exist a matrix algebra $\matrM{k}$ and a contractive completely positive (ccp) map $\phi \colon A \to \matrM{k}$ such that
\[
	\|\phi(ab) - \phi(a)\phi(b)\|\le \epsilon,
\]
for all  $a, b \in  F$, and
\[
	|\tr_{\matrM{k}}(\phi(a)) - \tau(a)| \le \epsilon,
\]
for any $a \in  F$.

\begin{theorem}\label{TWW}(Tikuisis, White, Winter \cite{TWW}, Cor. 6.1 + Cor. C) Let
$G$ be a discrete amenable group.
 Then every trace on $C^*(G)$ is quasidiagonal.
\end{theorem}

Recall that a \emph{character} of a group $G$ is a positive definite function on $G$ which is
constant on conjugacy classes and takes value 1 at the unit.

\begin{proposition}\label{MWSPversusMatricialStability}
Let $G$ be an amenable matricially stable group.
Then $G$ is Hilbert-Schmidt stable.
\end{proposition}
\begin{proof}
Let $\tau$ be a trace on $C^*(G)$.
By Theorem \ref{TWW} there exist almost multiplicative ccp maps $\phi_k \colon C^*(G) \to \matrM{n_k}$ such that
\[
	\tau(a) = \lim_{k\to \infty} \tr_{\matrM{n_k}}(\phi_k(a)),
\]
for each $a\in C^*(G)$.
Since $G$ is matricially stable, there exist $*$-homomorphisms $\psi_k \colon C^*(G) \to \matrM{n_k}$ such that
\[
	\|\psi_k(a) - \phi_k(a)\|\to 0
\]
and hence
\[
	\tau(a) = \lim_{k\to \infty} \tr_{\matrM{n_k}}(\psi_k(a)),
\]
for each $a\in C^*(G)$.
Since the restriction of any trace  of $C^*(G)$ onto $G$  is  a character, we conclude that each character of $G$ is a pointwise limit of traces of finite-dimensional representations. As was proved by D.~Hadwin and the second-named author in \cite{HadwinShulmanGroups}, for amenable groups  this condition is equivalent to Hilbert-Schmidt stability.
\end{proof}

\section{Open questions}

We list the following open questions and discuss them briefly.

\begin{enumerate}
\item Must a \SP\ group be finitely generated?
Must a \SP\ group be virtually free?\\[5mm]
\emph{In other words, are the sufficient conditions of Theorem \ref{vfissp} also necessary? We have little evidence to support this, but at least we have confirmed it in the abelian case (details will appear elsewhere). Also one may combine results from \cite{KRTW} and \cite{ed:cssc} to see that any torsion-free virtually abelian group (not necessarily finitely generated) which is \SP\ must be virtually free.}

\item Suppose $H$ is a finite index subgroup of $G$. Does  $H$ being \MWSP/\WSP/\SP $\;$ imply that $G$ is?\\[5mm]
\emph{Note that we have shown that the opposite direction fails for matricial stability. The fact that neither of the non-\MWSP\ wallpaper groups contain any of the \MWSP\ wallpaper groups by \cite[Table 4]{coxmos} may support this claim.}

\item When is $\baso(n, m)$ \MWSP/\WSP/\SP ? In particular is $\baso(n, m)$ matricially stable only when $n=\pm 1$, $m = \mp 1$?\\[5mm]
\emph{The metabelian case $n=1$, $m>1$ is particularly interesting, and indeed the structure of the associated $C^*$-algebra is rather well understood in this case, cf. \cite{brenken,pooyavalette}. We do not at present have any tools to address this case.}

\item Does matricial stability imply Hilbert-Schmidt stability?\\[5mm]
\emph{In \cite{HadwinShulman} D. Hadwin and the second-named author introduced a notion of matricial tracial stability for $C^*$-algebras which in the case of group $C^*$-algebras gives Hilbert-Schmidt stability.
The proof of Proposition \ref{MWSPversusMatricialStability}  can be modified to show that a separable nuclear matricially semiprojective $C^*$-algebra, all of whose quotients satisfy UCT, is matricially tracially stable.}

\item Can $C^*_r(G)$ be (matricially)(weakly) semiprojective when $G$ is non-amenable?\\[5mm]
\emph{
We have already provided examples of some \SP\ groups that are amenable and some that are not. In fact, when $G$ is not amenable we think it is unlikely that $C^*_r(G)$ is (matricially)(weakly) semiprojective.
 Even for the free group $\FF_2$,  $C^*_r(\FF_2)$ is not matricially semiprojective. Indeed it is MF by \cite{HaagerupThorbjornsen} and has no finite-dimensional representations, hence is not matricially semiprojective.}
\end{enumerate}

%%%%%%%%%%%%%%%%%%%%%%%%%%
%%% bibliography
%%%%%%%%%%%%%%%%%%%%%%%%%%

\end{document}